\documentclass[11pt,a4paper]{amsart}

\usepackage[utf8]{inputenc}
\usepackage[english]{babel}
\usepackage{stmaryrd,mathrsfs,bm,amsthm,mathtools,yfonts,amssymb,color}
\usepackage{graphicx}
\usepackage{enumitem}
\setlist[itemize]{leftmargin=*}
\usepackage{xcolor}
\usepackage{url}
\usepackage{accents}
\usepackage{bm}
\usepackage[all]{xy}
\usepackage{xcolor}
\usepackage{hyperref}
\hypersetup{
	colorlinks,
	linkcolor={blue!60!black},
	citecolor={green!60!black},
	urlcolor={green!60!black}
}
\usepackage{doi}
\usepackage[capitalise,nameinlink]{cleveref}
\crefformat{equation}{#2(#1)#3}
\crefrangeformat{equation}{#3(#1)#4 to~#5(#2)#6}
\crefmultiformat{equation}{#2(#1)#3}{ and~#2(#1)#3}{, #2(#1)#3}{ and~#2(#1)#3}
\usepackage{caption}
\usepackage{subcaption}
\usepackage{booktabs}
\usepackage{mathtools}
\usepackage{etoolbox}
\usepackage{overpic}
\usepackage{xspace}
\usepackage{comment}
\usepackage{chngcntr}
\usepackage{microtype}

\newcommand{\N}{\mathbb{N}}
\newcommand{\Z}{\mathbb{Z}}

\newcommand{\R}{\mathbb{R}}
\newcommand{\C}{\mathbb{C}}

\newcommand\restr[2]{{
  \left.\kern-\nulldelimiterspace
  #1
  \vphantom{\big|}
  \right|_{#2}
  }}

\newcommand{\de}{\partial}

\newcommand{\mz}{\frac{1}{2}}

\newcommand{\ang}[1]{\left\langle#1\right\rangle}
\newcommand{\uno}{\bm{1}}

\newcommand{\nin}{\not\in}

\newcommand{\weakto}{\rightharpoonup}

\newcommand{\mrestr}{\mathbin{\vrule height 1.6ex depth 0pt width
0.13ex\vrule height 0.13ex depth 0pt width 1.3ex}}

\renewcommand{\bar}{\overline}
\DeclareMathOperator{\sgn}{sgn}

\DeclareMathOperator{\diam}{diam}

\theoremstyle{definition}
\newtheorem{definition}{Definition}

\newtheorem{rmk}[definition]{Remark}

\newtheorem*{definition*}{Definition}
\newtheorem*{notazen*}{Notation}
\newtheorem*{rmk*}{Remark}
\newtheorem*{example*}{Example}
\newtheorem*{ack*}{Acknowledgement}
\newtheorem*{acks*}{Acknowledgements}
\newtheorem{thm}[definition]{Theorem}
\newtheorem{lemmaen}[definition]{Lemma}
\newtheorem{corollary}[definition]{Corollary}
\newtheorem{proposition}[definition]{Proposition}

\newtheorem*{thm*}{Theorem}
\newtheorem*{lemmaen*}{Lemma}
\newtheorem*{corollary*}{Corollary}
\newtheorem*{proposition*}{Proposition}
\newtheorem*{claim*}{Claim}
\newtheorem*{conj*}{Conjecture}

\usepackage[top=3.5cm,bottom=3.5cm,left=3.5cm,right=3.5cm]{geometry}
\DeclareMathOperator{\dive}{div}

\DeclareMathOperator{\spt}{spt}
\DeclareMathOperator{\dis}{Dis}
\DeclareMathOperator{\exc}{Exc}
\DeclareMathOperator{\mass}{\mathbb{M}}
\let\flat\undefined
\DeclareMathOperator{\flat}{\mathcal{F}}
\newcommand{\vol}{\mathrm{vol}}
\renewcommand{\ang}[1]{\langle #1\rangle}
\newcommand{\area}{\mathcal{H}^m(M)}
\newcommand{\fnorm}{\mathcal{F}}
\newcommand{\bp}{\mathbb{B}}
\newcommand{\tmz}{{\textstyle\frac{1}{2}}}
\renewcommand{\tfrac}[2]{{\textstyle\frac{#1}{#2}}}
\renewcommand{\epsilon}{\varepsilon}
\DeclareMathOperator{\Gr}{Gr}

\begin{document}
	\title[Minimal submanifolds as energy concentration sets]{non-degenerate minimal submanifolds \\ as energy concentration sets: \\ a variational approach}
	\author[G. De Philippis]{Guido De Philippis}
	\address{Courant Institute of Mathematical Sciences, New York University, 251 Mercer Street, New York, NY 10003, United States of America.}
	\email{guido@cims.nyu.edu}
	\author[A. Pigati]{Alessandro Pigati}
	\address{Courant Institute of Mathematical Sciences, New York University, 251 Mercer Street, New York, NY 10003, United States of America.}
	\email{ap6968@nyu.edu}
	
	
	
	\begin{abstract}
	We prove that every non-degenerate minimal submanifold of codimension two can be obtained as the energy concentration set of a family of critical maps for the (rescaled) Ginzburg--Landau functional. The proof is purely variational, and  follows the strategy laid out by Jerrard and Sternberg in \cite{JSte}, extending  a recent result for geodesics by Colinet--Jerrard--Sternberg \cite{CJS}. The same proof applies also to the $U(1)$-Yang--Mills--Higgs  and to the Allen--Cahn--Hilliard energies. 		
%

%
%
	\end{abstract}
	
	\maketitle
	\tableofcontents
	
	\section{Introduction}
	\subsection{General overview and statement of the problem}
	Minimal submanifolds are central objects  in geometric analysis and the calculus of variations.
	Their construction with variational methods is one of the main motivations for the entire field of geometric measure theory, and stimulated important developments
	in other parts of analysis and mathematics, particularly in the study of elliptic partial differential equations.
	
	While minimal hypersurfaces in $\R^n$ minimize the area locally, this needs not be true globally.
	The problem of finding submanifolds minimizing the area, subject to appropriate constraints, led to the successful theory of currents by Federer and Fleming,
	with important subsequent contributions concerning especially the regularity theory.
	On the other hand, we can hope to exhibit many more minimal submanifolds if we also look at the unstable ones, which can be viewed as saddle-type critical points for the area functional.
	
	Following the general scheme of min-max problems, initially proposed by Birkhoff for the construction of closed geodesic, Almgren and Pitts
	created a technical framework (the Almgren--Pitts theory) within geometric measure theory, capable of producing such (possibly) unstable minimal submanifolds
	in any given closed Riemannian manifold $N$, with a good regularity theory in the hypersurface case \cite{Pitts}.
	
	Minimal submanifolds also arise as concentration sets of various physically relevant energies, and the first rigorous mathematical formalization of this fact dates back to  an idea of De Giorgi: one can approximate the area of a hypersurface by means of ($\epsilon$-rescalings of) the Allen--Cahn--Hilliard  energy (see  \cref{AC}  below),
	defined on functions $u:N\to\R$ on a closed ambient $N$. Roughly speaking, as $\epsilon\to 0$, maps with bounded energy tend to take the values $1$ and $-1$, with a thin interface of width $\epsilon$ in between. This interface is expected to converge to a hypersurface whose area is the limit of the energy $E_\epsilon$.
	
	This heuristic idea was made into a deep principle, starting with the work of Modica and Mortola \cite{MM}, which showed that the Allen--Cahn energy $\Gamma$-converges to the $(n-1)$-area in an appropriate sense (see \cref{gamma.conv} below).
	Later, Hutchinson and Tonegawa \cite{HT} showed that the energy of critical maps always concentrates towards minimal hypersurfaces (rather, a weak version of them, given by stationary varifolds).
	Recently, Guaraco used this result to devise a new way to implement min-max schemes \cite{Guaraco}.
	
	In codimension two, similar (but weaker) results have been established by Lin--Rivi\`ere \cite{LR} and Bethuel--Brezis--Orlandi \cite{BBO} in the context of critical maps
	for the Ginzburg--Landau energy \cref{GL}.
	Among others, Jerrard--Soner \cite{JSoner} and Alberti--Baldo--Orlandi \cite{ABO} established the $\Gamma$-convergence theory, generalizing earlier results on the plane.
	Cheng \cite{Cheng} and Stern \cite{Stern} independently studied the asymptotic behavior of these energies on closed manifolds, and the latter work applied min-max constructions in order to produce minimal submanifolds (more precisely, stationary varifolds) of codimension two.
	
	Finally, the second-named author and Stern showed that in codimension two, minimal submanifolds can also  be approximated by the  Yang--Mills--Higgs energy on $U(1)$-bundles \cref{YMH}. More specifically, they showed the convergence of the energy density of critical points  to a stationary varifold \cite{PS} and, together with Parise, that the energy $\Gamma$-converges to the $(n-2)$-area \cite{PPS1}.
	The results obtained in \cite{PS} have a striking similarity with the ones of \cite{HT} and suggest that Yang--Mills--Higgs is a better way to approximate the area in codimension two, primarily thanks to a better decay of the energy away from vortices (we refer the reader to the introduction of \cite{PS} for additional discussion of this point).
	
This analysis can be extended also to the gradient flows of these energies, showing that they converge in a suitable sense to  the mean curvature flow \cite{Ilmanen, BOS, PPS2}.

The above results suggest that these energies are good ``diffuse'' approximations of the area functional, motivating the following natural question.

\begin{center}
\emph{Given a (smooth) minimal submanifold, can it be obtained as a (suitable)  limit of a sequence of  critical points of one  of the above energies? }
\end{center}

The answer is easily seen to be positive if the submanifold is an isolated  local minimizer of the area functional, via  standard  \(\Gamma\)-convergence arguments \cite{KS}.

For the  case of general critical points, the  answer was previously known only in few cases: in \cite{PR}   Pacard and Ritor\'e showed, via gluing techniques,  that  non-degenerate minimal hypersurfaces can be obtained as limits of critical points of the Allen--Cahn--Hilliard energy. In \cite{JSte} Jerrard and Sternberg settled a general variational approach to deal with similar questions via \(\Gamma\)-convergence techniques, and they showed that, in a ``mountain pass geometry'', there are critical points for the approximating functional with energy values close to the one of the limit critical point. Convergence of critical points cannot however been ensured in this general setting \cite[Remark~4.5]{JSte}.

In \cite{CJS},  Colinet, Jerrard and Sternberg  show that any non-degenerate geodesic in a three-manifold can be obtained as a limit of critical points of the Ginzburg--Landau energy.  The argument from \cite{CJS} combines the beautiful general framework developed in \cite{JSte} with the two following  observations, which are proved by carefully exploiting the one-dimensionality of the problem:

\begin{itemize}
\item we can identify a ``mountain pass geometry'' for the length functional around a non-degenerate geodesic, in the space of cycles with the flat norm;
\item non-degenerate geodesics are isolated among stationary varifolds (in a suitable sense).
\end{itemize}

The main result of this paper extends this approximation to every dimension  for all the energies described earlier  and provides a positive answer to the above question, completing the program started in \cite{JSte}. We loosely state it here, referring to \cref{thm.main} for a precise statement.

\begin{thm*}
Let \(N\) be a Riemannian manifold and \(M\subset N\) a non-degenerate minimal submanifold of codimension two (respectively, one). Then there exists a sequence of critical points of the Ginzburg--Landau or Yang--Mills--Higgs energies (respectively, of the Allen--Cahn--Hilliard energy) which ``converges'' to \(M\).
\end{thm*}

The strategy of the proof is based on  \cite{CJS}, together with the new key observation that the previous facts can be obtained in full generality by exploiting ideas from the regularity theory for minimal surfaces. This step, which is our main contribution, is crucial to show that,  while these properties can be easily shown  in a smooth topology, they are  in fact true in a much weaker one; we refer to \cref{s.main} for a more detailed description of the strategy of the proof.

	\subsection{Three energies approximating the area}
		In this paper we assume familiarity with basic geometric measure theory, for which the reader may consult \cite{Simon}.

	Let $(N^n,g)$ be a closed, oriented Riemannian manifold. For a (smooth) map $u:N\to\R$, the Allen--Cahn energy is defined as
	\begin{align}\label{AC}
		&E_\epsilon(u)
		:=c\epsilon\int_N\Big(\frac{|du|^2}{2}+\frac{(1-u^2)^2}{4\epsilon^2}\Big),\quad c=\frac{3}{2\sqrt{2}},
	\end{align}
	where the normalization constant is $c(W)=1\,/\int_{-1}^1\sqrt{2W(t)}\,dt$
	for general double-well potentials $W$ (here we choose $W(u)=\frac{(1-u^2)^2}{4}$).
	
	The Ginzburg--Landau energy for maps $u:N\to\C$, in the version considered here, is	
	\begin{align}\label{GL}
		&E_\epsilon(u)
		:=\frac{1}{\pi\log(\epsilon^{-1})}\int_N\Big(\frac{|du|^2}{2}+\frac{(1-|u|^2)^2}{4\epsilon^2}\Big).
	\end{align}
	
	Finally, given a complex line bundle $L\to N$ endowed with a Hermitian metric, the Yang--Mills--Higgs energy is defined on couples $(u,\nabla)$,
	where $u:N\to L$ is a section and $\nabla$ is a metric connection on $L$. It is given by
	\begin{align}\label{YMH}
		&E_\epsilon(u,\nabla)
		:=\frac{1}{2\pi}\int_N\Big(|\nabla u|^2+\frac{(1-|u|^2)^2}{4\epsilon^2}+\epsilon^2|F_{\nabla}|^2\Big),
	\end{align}
	where $F_\nabla$ is the curvature of $\nabla$, taking values into $i\R$ (the Lie algebra of $U(1)$). In the sequel, as in \cite{PS}, we will often use the closed real-valued two-form
	\begin{align*}
		&\omega=\omega_\nabla:=iF_\nabla.
	\end{align*}
	When dealing with this energy, it will be convenient to fix a smooth reference (metric) connection $\nabla_0$; any other (metric) connection can be written as $\nabla=\nabla_0-i\alpha$,
	for a real-valued $\alpha\in\Omega^1(N)$, so that $\omega_\nabla=\omega_0+d\alpha$ (with $\omega_0:=\omega_{\nabla_0}$).
	
	For any of these energies, we denote by $\mu(u)$ (or $\mu(u,\nabla)$) the energy density, which is a Radon measure on $N$.
	For instance, for Ginzburg--Landau we let
	\begin{align*}
		&\mu(u)
		:=\frac{1}{\pi\log(\epsilon^{-1})}\Big(\frac{|du|^2}{2}+\frac{(1-|u|^2)^2}{4\epsilon^2}\Big)\,\mathcal{H}^n.
	\end{align*}
	
	In order to state the $\Gamma$-convergence results, we also introduce the following Jacobian quantities, which are initially defined as differential forms but are identified with their dual currents throughout the paper:
	namely, for such a Jacobian $J\in\Omega^k(N)$, we define the corresponding $(n-k)$-current (still called $J$) by the assignment
	\begin{align*}
		&\ang{J,\eta}:=\int_N J\wedge\eta,
	\end{align*}
	for all $\eta\in\Omega^{n-k}(N)$.
	
	When dealing with the Allen--Cahn energy \cref{AC} we define
	\begin{align*}
		&J(u)
		:=\frac{3}{4}(1-u^2)\,du
	\end{align*}
	(or $\sqrt{W(u)}\,du\,/\int_{-1}^1\sqrt{W(t)}\,dt$ for general double-well potentials $W$),
	for Ginzburg--Landau \cref{GL} we define
	\begin{align*}
		&J(u)
		:=\frac{1}{\pi}du^1\wedge du^2
	\end{align*}
	(for the two components $u=u^1+iu^2$), and for Yang--Mills--Higgs \cref{YMH} we let
	\begin{align}\label{jac.ymh}\begin{aligned}
		J(u,\nabla=\nabla_0-i\alpha)
		&:=\frac{1}{2\pi}(d\ang{\nabla u,iu}+\omega_\nabla) \\
		&\phantom{:}=\frac{1}{2\pi}d(\ang{\nabla_0 u,iu}+(1-|u|^2)\alpha)+\frac{\omega_0}{2\pi},
	\end{aligned}\end{align}
	following the convention of \cite{PS} that the Hermitian metric $\ang{\cdot,\cdot}$ on the line bundle is real-valued (or, equivalently, is the real part of a complex-valued sesquilinear form).
	
	The following result is an instance of the principle that these energies approximate the area functional.
	It can be regarded as the $\Gamma$-convergence of $E_\epsilon$ to the area; more precisely it is the so-called ``liminf inequality'' part of $\Gamma$-convergence. The other part of $\Gamma$-convergence (involving recovery sequences)
	also holds in an appropriate sense, but we omit the precise statement since it will not be needed in this paper.
	Rather, we will need the existence of a recovery sequence for our smooth minimal submanifold $M$ (see \cref{prop.constr} below).
	
	\begin{thm}\label{gamma.conv}
		The following statements hold for an arbitrary sequence $\epsilon_j\to 0$:
		\begin{itemize}
			\item given a sequence of maps $u^j:N\to\R$ with $|u^j|\le 1$ and $\sup_j E_{\epsilon_j}(u^j)<\infty$, for the Allen--Cahn energy \cref{AC},
			there exists a subsequence such that
			\begin{align*}
				&\mu(u^j)\weakto\mu,\quad \fnorm(J(u^j)-S)\to 0,\quad |S|\le\mu
			\end{align*}
			for a suitable integral $(n-1)$-cycle $S$ and a Radon measure $\mu$ (actually, $S$ is the boundary of a finite perimeter set);
			\item the same holds for Ginzburg--Landau \cref{GL}, with a limit integral $(n-2)$-cycle $S$;
			\item the same holds for Yang--Mills--Higgs \cref{YMH} (for a sequence of couples $(u^j,\nabla^j)$), with a limit integral $(n-2)$-cycle $S$.
		\end{itemize}
	\end{thm}
	In the previous theorem, as well as the next ones, we assume $n\ge 1$ in the first case and $n\ge 2$ in the second and third cases.
	
	In the statement, $\fnorm$ is the flat norm on $m$-boundaries, defined as
	\begin{align*}
		&\fnorm(S):=\inf\{\mass(R)\mid \de R=S\},
	\end{align*}
	where $R$ ranges among all $(m+1)$-currents with boundary $S$.
	Observe that, while $J(u)$ is always a boundary in the Allen--Cahn and Ginzburg--Landau settings,
	$J(u,\nabla)$ is only a cycle in general; on the other hand, in view of \cref{jac.ymh},
	the real homology class $[J(u,\nabla)]=[\frac{\omega_0}{2\pi}]\in H_{n-2}(N;\R)$ is fixed.
	
	For Allen--Cahn the standard (but hard to access) reference is \cite{MM}.\footnote{In this setting the proof can be sketched as follows:
		$J(u^j)$ equals $d(F(u^j))/F(1)$, where $F(t):=\int_{-1}^t\sqrt{W(s)}\,ds$ (here, $W(s)=\frac{(1-s^2)^2}{4}$).
		Since $J(u^j)$ is bounded pointwise by the energy density, $F(u^j)$ is bounded in $BV$ and, hence, \(\{u^j\}\)  has a subsequential strong limit $u^\infty$ in $L^1$. Since $\int_N W(u^j)\to 0$, we obtain $W(u^\infty)=0$ (a.e.), or equivalently $u^\infty(x)\in\{\pm 1\}$ for a.e.\ $x$.
		Thus, $J(u^j)$ converges in the flat norm to $d(F(u^\infty))/F(1)$, the boundary of $N\setminus\{u^\infty=1\}=\{u^\infty=-1\}$.}
	We refer the reader to \cite[Theorem~1.1.(i)]{ABO} or \cite[Theorem~5.2]{JSoner} for the proof in the Ginzburg--Landau setting\footnote{The work \cite{ABO} shows convergence in the flat norm, in the Euclidean setting. Applying this result locally and using a partition of unity, one obtains $J(u^j)-S=Q^j+\de R^j$ (along a subsequence),
	for an integral cycle $S$ and currents $Q^j,R^j$ of infinitesimal mass; the argument from \cite[p.~588]{GMS} then allows to write $Q^j=\de\tilde R^j$, with $\mass(\tilde R^j)\le C(N)\mass(Q^j)$, so that $\fnorm(J(u^j)-S)\to 0$.}
	and to \cite[Theorem~1.2.(i)]{PPS1} for Yang--Mills--Higgs (convergence of the Jacobians in the flat norm is a consequence of the proof from \cite{PPS1}).

	The theorem below, which is substantially more difficult to prove (especially for Ginzburg--Landau), concerns the structure of the limit measure $\mu$ when the sequence consists of \emph{critical} maps (or couples).
	
	\begin{thm}\label{asymp}
		In the same situation of the previous theorem, if the maps $u^j$ (or couples $(u^j,\nabla^j)$) are also critical for $E_{\epsilon_j}$, then the following holds:
		\begin{itemize}
			\item in the Allen--Cahn setting, $\mu=|V|$ for an integral stationary $(n-1)$-varifold $V$;
			\item for Ginzburg--Landau, $\mu=|V|+\mz|h|^2\,\mathcal{H}^n$ for a stationary rectifiable $(n-2)$-varifold $V$
			with density in $\{1\}\cup[2,\infty)$ a.e.\ and a harmonic one-form $h\in\Omega^1(N)$;
			\item for Yang--Mills--Higgs, $\mu=|V|$ for an integral stationary $(n-2)$-varifold $V$.
		\end{itemize}
	\end{thm}

	Proofs were given in \cite{HT} (see also \cite[Appendix~B]{Guaraco} for comments on the manifold setting), \cite{Cheng,Stern} (in the manifold case, which requires slightly more care than the Euclidean case, relying on the earlier works \cite{BBO,LR}) and \cite{PS}, respectively; see also the recent work \cite{PS2} for the characterization of the density in the Ginzburg--Landau setting.

	We also have a parabolic version of the previous theorem, namely an asymptotic result for the gradient flow of the energy as $\epsilon\to 0$.
	The gradient flow equations for the three energies are
	\begin{align}\label{gf.ac}
		&\de_t u=\Delta u+\frac{(1-|u|^2)u}{\epsilon^2}
	\end{align}
	for the Allen--Cahn and Ginzburg--Landau functionals (more precisely, for the energy prior normalization, namely $\frac{2\sqrt{2}}{3}\epsilon^{-1}E_\epsilon$
	and $\pi\log(\epsilon^{-1})E_\epsilon$, respectively),
	with respect to the usual $L^2$-scalar product,	and
	\begin{align}\label{gf.ymh}
		\left\{
		\begin{aligned}
		\partial_t u&=-\nabla^*\nabla u+{\textstyle\frac{1}{2\epsilon^2}}(1-|u|^2)u, \\
		\partial_t\alpha&=-d^*(\omega_0+d\alpha)+{\textstyle\frac{1}{\epsilon^2}}\langle iu,\nabla u\rangle, 
		\end{aligned}
		\right.
	\end{align}
	for the Yang--Mills--Higgs functional (more precisely, for $2\pi E_\epsilon$), with respect to the scalar product $\int_N(\ang{u,v}+\epsilon^2\ang{\alpha,\beta})$ on the space of couples $(u,\nabla)=(u,\nabla_0-i\alpha)$.
	
	Long-term existence, uniqueness and continuous dependence on the (smooth) initial condition $u_0$ (or $(u_0,\nabla_0)$) are standard for \cref{gf.ac}, and are detailed in \cite[Section~6.3]{PPS1}
	for \cref{gf.ymh} (with the natural assumption $|u_0|\le 1$).

	\begin{thm}\label{asymp.flow}
		Given a sequence of smooth initial data $u^j_0$
		with $\sup_j E_{\epsilon_j}(u^j_0)\le\Lambda<\infty$ and $|u^j_0|\le 1$ (for a sequence $\epsilon_j\to 0$),
		up to a subsequence we have $\mu(u^j_t)\weakto\mu_t$ for all $t\ge 0$ (maps should be replaced by couples for Yang--Mills--Higgs),
		for a family of Radon measures $(\mu_t)_{t\in[0,\infty)}$ such that the following holds:
		\begin{itemize}
			\item for Allen--Cahn, $(\mu_t)_{t\ge 0}$ is an $(n-1)$-dimensional Brakke flow and $\mu_t$ is the weight of an integral varifold for a.e.\ $t$;
			\item for Ginzburg--Landau, $\mu_t=\nu_t+\mz|\alpha_t|^2$ for an $(n-2)$-dimensional Brakke flow $(\nu_t)_{t\in[0,\infty)}$
			and a closed one-form $\alpha_t$ solving the heat equation $\de_t\alpha_t=-\Delta_H\alpha_t$
			(where $\Delta_H\alpha_t=dd^*\alpha_t$ is the Hodge Laplacian); also,
			$\nu_t$ is the weight of a rectifiable varifold with $(n-2)$-density bounded below by a constant $\eta(t,N,\Lambda)>0$, for a.e.\ $t>0$;
			\item for Yang--Mills--Higgs, $(\mu_t)_{t\ge 0}$ is an $(n-2)$-dimensional Brakke flow and $\mu_t$ is the weight of an integral varifold for a.e.\ $t$.
		\end{itemize}
	\end{thm}

	For the proof, we refer the reader to \cite{Ilmanen,Tonegawa} (which easily extend to manifolds), \cite{BOS} (see also \cite{Colinet} for the manifold setting) and \cite{PPS2}, respectively.
	We refer to \cite[Sections~1.8--1.9]{Ilmanen} for the definition of Brakke flow.
	
	\begin{rmk}
		The constant $\eta(t,N,\Lambda)$ obtained in \cite{BOS} is a continuous function of $t>0$.
	\end{rmk}
	
	\subsection{Main theorem, strategy and auxiliary results}\label{s.main}
	With the above notation in place, our main result now reads as follows.
	\begin{thm}\label{thm.main}
		Given a (closed, smooth, embedded, oriented) non-degenerate minimal submanifold $M^m$ of $N^n$,
		there exists a family of
		\begin{itemize}
		\item critical maps $(u_\epsilon)$ for \cref{AC} when $m=n-1$ and $M$ bounds a smooth open set,
		\item critical maps $(u_\epsilon)$ for \cref{GL} when $m=n-2$ and $M$ bounds an oriented hypersurface,
		\item critical couples $(u_\epsilon,\nabla_\epsilon)$ for \cref{YMH}
		when $m=n-2$ and $M$ is oriented, for an appropriate line bundle $L\to N$ (specifically, such that the Euler class $c_1(L)\in H^2(N;\Z)$ is Poincar\'e dual to $\llbracket M\rrbracket\in H_{n-2}(N;\Z)$),
		\end{itemize}
		such that the energy density
		\begin{align*}
			&\mu(u_\epsilon)\weakto\mathcal{H}^m\mrestr M
		\end{align*}
		in the sense of Radon measures ($\mu(u_\epsilon,\nabla_\epsilon)\weakto\mathcal{H}^m\mrestr M$ for Yang--Mills--Higgs).
	\end{thm}
	Note that, since $M$ is non-degenerate, we must necessarily have $m\ge 1$.
	In the statement, the family is parametrized by $\epsilon\in(0,\bar\epsilon)$, for an appropriate $\bar\epsilon>0$ depending on $M$ and $N$.
	In the Allen--Cahn and Ginzburg--Landau cases, we endow $M$ with the boundary orientation. Thus, in all cases the oriented submanifold $M$ gives rise to a well-defined integral $m$-cycle $\llbracket M\rrbracket$. It should be noted that, in fact, only the orientability of $M$ and $N$ plays a role; the other topological conditions could be dropped by using sections of suitable (real or complex) line bundles over $N$.
	
	\begin{rmk}
		The non-degeneracy assumption cannot be dropped entirely.
		For instance, on $S^1=\R/\Z$ it is easy to check (due to the explicit nature of the solutions) that two points $M=\{p,q\}$ can be realized as the concentration set in the Allen--Cahn setting
		only when they are antipodal.\footnote{We wish to thank Christos Mantoulidis for pointing out this simple example.} Similar examples should be expected also in higher dimension.
	\end{rmk}
	
	As explained earlier, we follow the same strategy of \cite{CJS},  the main difficulty being to obtain \cref{thm.core2} and \cref{thm.core} below in arbitrary dimension. In particular, our main main theorem follows easily from the following one (cf.\ \cite[Proposition~4.1]{CJS}).

	\begin{thm}\label{thm.crit}
		Given $\delta>0$, there exists $\epsilon_0(M,\delta)>0$ such that the following holds.
		Given any time $T>0$, for $\epsilon\in(0,\epsilon_0)$ there exists a solution $(u^\epsilon_t)$ or $(u^\epsilon_t,\nabla^\epsilon_t)$ to the gradient flow equation for $E_\epsilon$ (chosen among the three energies) such that
		\begin{align}\label{crit.bds}
			&\fnorm(J(u^\epsilon_t)-\llbracket M\rrbracket )\le\delta,
			\quad |E_\epsilon(u^\epsilon_t)-\mathcal{H}^m(M)|\le\delta,
		\end{align}
		for all $t\in[0,T]$ (with $u^\epsilon$ replaced by $(u^\epsilon,\nabla^\epsilon)$ for Yang--Mills--Higgs).
	\end{thm}

	Indeed, since the time $T$ is arbitrarily large, we can select an intermediate time such that $\de_t u^\epsilon_t$ is arbitrarily small in $L^2$,
	and obtain a critical map in the limit. From the second part of \cref{crit.bds} we only use the information that $E_\epsilon(u^\epsilon_t)\le\mathcal{H}^m(M)+\delta$.
	Then \cref{gamma.conv} easily implies that $M$ is the energy concentration set (since we can make $\delta\to 0$ as $\epsilon\to 0$).

	The strategy to obtain \cref{thm.crit} can be outlined as follows: we define a finite-dimensional family $(M_w)_{w\in\bp}$ of deformations of $M$ with lower area,
	parametrized by a ball $\bp\subset\R^\ell$ (with $M_0=M$), where $\ell$ (the Morse index of $M$) measures the instability of $M$.
	Then we create a min-max geometry: we introduce a function $P$ projecting the set of $m$-submanifolds (or cycles) to $\R^\ell$,
	with $P(M_w)=w$ (hence, $P(M_0)=0$), in such a way that the area of $M'$ is always larger than the area of $M$ whenever $P(M')=0$ and $M'$ is close to $M$
	(in a weak sense). Thus, $\{P(\cdot)=0\}$ can be regarded as the high-energy obstacle in the mountain pass situation, and $M$ becomes the precise location of the mountain pass.

	More precisely, for the map $P$ constructed in the next section, the following theorem holds. It shows a stronger statement, initially proved in another version by White \cite{White},
	namely that $M$ is a strict local minimizer for the perturbed area functional $\tilde\mass(S):=\mass(S)+\lambda|P(S)|^2$ (for a suitably large $\lambda>0$).
	For this functional, $M$ has a positive second variation; the difficulty (and usefulness) of the result lies in the fact that minimality holds in a \emph{flat} neighborhood of $M$. We also mention that a very similar statement is proved, in dimension one, in the PhD thesis of Mesaric \cite{Mesa}.\footnote{We wish to thank Robert Jerrard for pointing this out to us.}

	\begin{thm}\label{thm.core2}
	The cycle $\llbracket M\rrbracket $ is a strict local minimizer for $\tilde\mass$ among integral $m$-cycles, in a neighborhood for the flat norm. More precisely, there exists $\delta_1(M)>0$ such that, for any integral $m$-cycle $S$ with $\fnorm(S-\llbracket M\rrbracket )\le\delta_1$, one has
	\begin{align*}
	&\tilde\mass(S)\ge\tilde\mass(\llbracket M\rrbracket )=\area,
	\end{align*}
	with equality only if $S=\llbracket M\rrbracket $.
	\end{thm}

	We have the function $P$, projecting cycles to $\R^\ell$, but we also have the Jacobian $J(\cdot)$, which maps any function $u$ (or couple $(u,\nabla)$) to a cycle.
	The idea is to use $J(\cdot)$ and $P$ to find approximate critical points for $E_\epsilon$, while keeping track of their distance (via $J(\cdot)$) from the cycle $\llbracket M\rrbracket $.
	
	Specifically, once we have reproduced the mountain pass situation, we approximate each $M_w$ with a map $u^{\epsilon,w}_0$ (or a couple $(u^{\epsilon,w}_0,\nabla^{\epsilon,w}_0)$)
	with energy close to the area of $M_w$ and $J(u^{\epsilon,w}_0)$ close to $\llbracket M_w\rrbracket $.
	Then, by applying the gradient flow for a time $T$, we get new maps $u^{\epsilon,w}_T$, and a degree-theoretic argument
	shows that there exists a $w\in\bp$ such that $P(J(u^{\epsilon,w}_T))=0$, with $J(u^{\epsilon,w}_T)$ in the neighborhood where \cref{thm.core2} applies.
	In particular, using also \cref{gamma.conv}, the energy of this map at time $T$ is at least the area of $M$ (up to an infinitesimal error), meaning that
	the energy almost did not change along the flow.
	
	Actually, the aforementioned argument uses again \cref{thm.core2}, together with $\Gamma$-convergence.	
	Roughly speaking, the idea is that at each time $t\in[0,T]$ there must be some $w$ such that $P(J(u^{\epsilon,w}_t))=0$.
	If at some time the Jacobian $J(u^{\epsilon,w}_t)$ escapes the weak (closed) neighborhood of $\llbracket M\rrbracket $ where \cref{thm.core2} holds,
	assuming $t$ is the first time when this happens, in the limit $\epsilon\to 0$ we get a cycle $S$ with $P(S)=0$ and area bounded by $\area$,
	on the boundary of the neighborhood, contradicting \cref{thm.core2}.
	
	Finally, we have to ensure that $J(u^{\epsilon,w}_t)$ stays close to $\llbracket M\rrbracket $ for all times $t\in[0,T]$ (for a fixed $T$ and with $w$ chosen as above).
	This can be proved for $t=T$ with the same argument, but in order to deal with all the interval another idea is needed.
	We look at the first time when the Jacobian reaches a certain positive distance from $\llbracket M\rrbracket $. Since the energy along the flow is almost unchanged,
	we have a stationary situation in the limit $\epsilon\to 0$, and we obtain a contradiction to \cref{thm.core} below, which in turn follows easily from the next rigidity result.
	
	\begin{thm}\label{thm.core.main}
		Let \(M\) and \(N\) be as above. Given $\eta,\Lambda>0$, there exists $\delta(M,N,\eta,\lambda)\in(0,\rho)$ with the following property:
		if a stationary rectifiable $m$-varifold $V$ is such that
		\begin{itemize}
		\item $d_H(\spt|V|,M)\le\delta$ (with $d_H$ the Hausdorff distance),
		\item $\mathbb{M}(V)\le\Lambda$ (with $\mathbb{M}(V)$ the total mass of $V$),
		\item $\Theta^m(|V|,y)\ge\eta$ for all $y\in\spt|V|$,
		\item $\spt|V|\cap\pi^{-1}(x)$ is nonempty for all $x\in M$, and consists of a single point for $x\in M\setminus E$, for a set $E\subseteq M$ with $\mathcal{H}^m(E)\le\delta$,
		\end{itemize}
		then $\spt|V|=M$.
	\end{thm}
	
	In the statement, $\pi:\bar B_\delta(M)\to M$ denotes the nearest point projection.
	
	\begin{corollary}\label{thm.core}
		Given $\eta>0$, there exists $\delta_2(M,\eta)>0$ with the following property:
		there cannot exist a stationary rectifiable $m$-varifold $V$ with density at least $\eta$ and an integral $m$-cycle $S$
		satisfying
		\begin{align*}
		&0<\fnorm(S-\llbracket M\rrbracket )\le\delta_2,\quad |S|\le|V|,\quad\mass(V)\le\mathcal{H}^m(M)+\delta_2.
		\end{align*}
	\end{corollary}
	
	This final ingredient is the main novelty of the present paper. The main strategy to prove it is to parametrize $V$ as a multigraph over $M$
	and construct an approximate Jacobi field out of it, contradicting the assumption that $M$ is non-degenerate minimal.
	
	\begin{rmk}
	Actually, in the Allen--Cahn and Yang--Mills--Higgs settings, we just need to use \cref{thm.core} with $\eta=1$. In this case, this theorem can be easily deduced from Allard's classical regularity theorem \cite{Al}. However, in the Ginzburg--Landau setting, since the sharp density lower bound in \cref{asymp.flow} is not known, we need this result in its full strength.
	We expect that techniques similar to those used in \cite{PS2} allow to obtain $\eta=1$ also in this parabolic setting. However, our machinery does apply to any energy for which $\Gamma$-convergence and convergence of the gradient flows hold (with a positive density lower bound $\eta>0$ for the latter).
	\end{rmk}
	
%
%
	
	\addtocontents{toc}{\protect\setcounter{tocdepth}{0}}
	\subsection*{Acknowledgements.}
	The work of G.D.P is  partially supported by the NSF grant DMS 2055686 and by the Simons Foundation.  The authors would like to thank Fanghua Lin and Tristan Rivi\`ere   for their interest and for several inspiring  discussions.
	\addtocontents{toc}{\protect\setcounter{tocdepth}{2}}
	
	\section{Definition of the family $(M_w)$ and the perturbed mass functional}\label{sec.prel}
	
	Given a (smooth, embedded, oriented) non-degenerate minimal submanifold $M^m$ in $N^n$,
	we denote by $J$ the symmetric bilinear form representing the second variation of the area,
	defined on smooth sections of the normal bundle $T^\perp M$. Namely, given a smooth section $\varphi:M\to T^\perp M$,
	defining $M_\varphi$ to be the image of the map $x\mapsto\exp_x(\varphi(x))$ (from $M$ to $N$), we have
	\begin{align}\label{second.var.def}
		&\mathcal{H}^m(M_{s\varphi})=\mathcal{H}^m(M)+\frac{s^2}{2}J(\varphi,\varphi)+o(s^2)
	\end{align}
	(the linear part vanishes by minimality of $M$).
	We can write
	\begin{align*}
		&J(\varphi,\psi)=\int_M\ang{J\varphi,\psi}
	\end{align*}
	for a self-adjoint elliptic second-order operator (still denoted $J$), called the \emph{Jacobi operator} of $M$.
	The assumption that $M$ is non-degenerate means that $J$ has trivial kernel; for generic metrics $g$ on the ambient $N$, this assumption is always met for all minimal submanifolds \cite{White2}.
	
	Let $\lambda_1\le\lambda_2\le\dots$ be the eigenvalues of $J$, listed with multiplicity,
	and select an orthonormal basis $(\varphi_j)_{j=1}^\infty$ of corresponding eigenvectors.

	\begin{definition}
		We fix $\lambda>0$ such that $\mz\lambda_1+\lambda>0$.
		Also, we let $\ell$ be the maximum index such that $\lambda_\ell<0$.
		This finite number $\ell\in\N$ is called the \emph{Morse index} of $M$ (possibly $\ell=0$).
	\end{definition}
	
	We fix $\rho_0>0$ small, such that the exponential map
	\begin{align*}
		&\exp:\{v\in T^\perp M:|v|\le\rho_0\}\to \bar B_{\rho_0}(M)
	\end{align*}
	is a diffeomorphism ($\bar B_{\rho_0}(M)$ is the tubular neighborhood consisting of all the points having distance at most $\rho_0$ from $M$).
	We call $\pi:\bar B_{\rho_0}(M)\to M$ the nearest point projection and $d_M$ the distance function to $M$. For $y\in \bar B_{\rho_0}(M)$ we have
	\begin{align*}
		&y=\exp_x(v)
	\end{align*}
	for $x=\pi(y)$ and a suitable $v\in T_x^\perp M$ with $|v|=d_M(y)$ ($x\in M$ and $v\in T_x^\perp M$ are unique, under the constraint $|v|\le\rho_0$).
	
	We extend each $\varphi_j$ to a section $\varphi_j:\bar B_{\rho_0}(M)\to TN$ via parallel transport along geodesic rays orthogonal to $M$. Then we let
	\begin{align*}
		&\omega_j:=\ang{X,\varphi_j}\,\pi^*\vol_M,
	\end{align*}
	where $\vol_M$ is the volume form of $M$ and $X:=\operatorname{grad}\frac{d_M^2}{2}$.
	Finally, we extend each $\omega_j$ to a smooth $m$-form on all of $N$, in an arbitrary way.
	
	\begin{definition}
		For any $m$-current $T$ we define
		\begin{align}\label{P.def}
			&P(T):=(\ang{T,\omega^1},\dots,\ang{T,\omega^\ell}).
		\end{align}
		The modified mass functional is then
		\begin{align}\label{mod.mass.def}
			&\tilde\mass(T):=\mass(T)+\lambda|P(T)|^2.
		\end{align}
	\end{definition}
	
	Note that if $y=\exp_x(v)$, for some $v\in T_x^\perp M$ with norm $|v|\le\rho_0$, then $\ang{X,\varphi_i}(y)=\ang{v,\varphi_i(x)}$; hence,
	\begin{align}\label{P.graph}
		&P(\llbracket M_\varphi\rrbracket )=\Big(\int_M\ang{\varphi,\varphi_j}\Big)_{j=1}^\ell.
	\end{align}
	It is straightforward to check that \cref{second.var.def} can be improved to
	\begin{align}\label{second.var.bd}
		&\mathcal{H}^m(M_\varphi)
		=\mathcal{H}^m(M)+\mz J(\varphi,\varphi)+O(\|\varphi\|_{C^1}\|\varphi\|_{W^{1,2}}^2)
	\end{align}
	for a section $\varphi$ with $|\varphi|\le\rho_0$,
	where the two norms are the $C^0$-norm and the $L^2$-norm of $(\varphi,\nabla^\perp\varphi)$, respectively.
	From \cref{P.graph}, \cref{second.var.bd} and our choice of $\lambda$, it follows that
	\begin{align}\label{second.var.bd.cons}
		&\tilde\mass(\llbracket M_\varphi\rrbracket )>\mathcal{H}^m(M)
	\end{align}
	whenever $\varphi$ is nontrivial and small enough in the $C^1$-norm. The content of \cref{thm.core2} is that
	this inequality still holds for cycles close to $M$ in a much weaker sense, for the same choice of $\lambda$.
	
	\begin{definition}
		Given $w\in\R^\ell$ small enough, we let $\varphi_w:=\sum_{j=1}^\ell w_j\varphi_j$ and
		\begin{align*}
			&M_w:=M_{\varphi_w}.
		\end{align*}
		We fix a \emph{closed} ball $\bp$ in $\R^\ell$, centered at the origin, such that
		$|\sum_{i=1}^\ell w_i\varphi_i|<\rho_0$ for all $w\in\bp$ (in particular, $M_w$ is a smooth embedded submanifold since by elliptic regularity each $\varphi_j$ is smooth)
		and
		\begin{align}\label{smaller.area}
			&\mathcal{H}^m(M_w)<\mathcal{H}^m(M)\quad\text{for all }w\in\bp\setminus\{0\},
		\end{align}
		as well as
		\begin{align}\label{M.w.flat}
			&\fnorm(\llbracket M_w\rrbracket -\llbracket M\rrbracket )\le\mz\delta_1,
		\end{align}
		where $\delta_1$ is the constant from \cref{thm.core2}.
	\end{definition}
	
	Note that, in the previous definition,  property \cref{smaller.area} can be guaranteed using the fact that each $\varphi_j$ corresponds to a negative eigenvalue for the second variation of the area.
	
	In the remainder of this section, we outline a proof of the following proposition. For a single submanifold, it is a special case of the existence of recovery sequences
	in the $\Gamma$-convergence of our energies to the area. Some care is required in order to obtain a family which depends continuously on $w$.
	
	\begin{proposition}\label{prop.constr}
		Under the assumptions of \cref{thm.main},
		for $\epsilon$ small enough there exists a continuous family $(u^{\epsilon,w}_0)_{w\in\bp}$ such that
		\begin{align}\label{constr.bds}
		&E_\epsilon(u^{\epsilon,w}_0)
		\le \mathcal{H}^m(M_w)+\delta(\epsilon),
		\quad
		\fnorm(J(u^{\epsilon,w}_0)-\llbracket M_w\rrbracket )
		\le \delta(\epsilon),
		\end{align}
		for a quantity $\delta(\epsilon)$ which is infinitesimal as $\epsilon\to 0$.
	\end{proposition}
	
	In the statement, continuity refers to the smooth topology on maps (or couples);
	also, $u^{\epsilon,w}_0$ should be replaced with a couple $(u^{\epsilon,w}_0,\nabla^{\epsilon,w}_0)$ in the Yang--Mills--Higgs setting.
	
	\begin{rmk}\label{gamma.rmk}
		Actually, the maps (or couples) satisfy $\mu(u^{\epsilon,w}_0)\weakto\mathcal{H}^m\mrestr M_w$ for every $w\in\bp$.
		This follows from the proof and can also be deduced from the $\Gamma$-convergence statement (\cref{gamma.conv},
		where $S=\llbracket M_w\rrbracket $ and where the inequality $|S|\le\mu$ must be an equality, as $S$ and $\mu$ have the same total mass).
	\end{rmk}
	
	\begin{proof}[Proof of \cref{prop.constr}]
		Given a $k$-form $\omega$ and $y\in\bar B_{\rho_0}(M)$, let
		\begin{align*}
			&|\omega|_h^2(y):=\sum_{i=1}^{m}|\iota_{e_i}\omega|^2,
		\end{align*}
		where $\{e_i\}_{i=1}^{m}$ is an orthonormal basis of $(\operatorname{ker}d\pi(y))^\perp$ and \(\iota_{e_i}\omega\) is the interior product between the \(k\)-form and the vector \(e_i\).
		Thus, roughly speaking, $|\omega|_h(y)$ measures the components of $\omega$ which are not perpendicular to $M$.
		
		We claim that it is enough to find maps $u^{\epsilon,0}_0$ (or couples $(u^{\epsilon,0}_0,\nabla^{\epsilon,0}_0)$)
		such that the statement holds for $w=0$ and such that
		\begin{itemize}
			\item $\int_{B_{\rho_0}(M)}|du^{\epsilon,0}_0|_h^2=o(\epsilon^{-1})$ for Allen--Cahn;
			\item $\int_{B_{\rho_0}(M)}|du^{\epsilon,0}_0|_h^2=o(\log(\epsilon^{-1}))$ for Ginzburg--Landau;
			\item $\int_{B_{\rho_0}(M)}(|\nabla^{\epsilon,0}_0 u^{\epsilon,0}_0|_h^2+\epsilon^2|F_{\nabla^{\epsilon,0}_0}|_h^2)=o(1)$ for Yang--Mills--Higgs.
		\end{itemize}
		Once this is done, we can take a family of diffeomorphisms $\Phi_w:N\to N$ (depending smoothly on $w\in\bp$) such that $\Phi_0=\operatorname{id}$, $\Phi_w(M)=M_w$ and
		$d\Phi_w(x)$ maps the normal space $T^\perp_x M$ to the normal space $T^\perp_{\Phi_w(x)} M_w$ isometrically, for all $x\in M$. Then we can simply define
		$u^{\epsilon,w}_0:=u^{\epsilon,0}_0\circ\Phi_w^{-1}$ for Allen--Cahn or Ginzburg--Landau.
		Our assumption on the differential of $\Phi_w$ gives
		\begin{align}\label{par.dist}
			&|du^{\epsilon,w}_0|(\Phi_w(y))
			=|du^{\epsilon,0}_0|(y)+O(|du^{\epsilon,0}_0|_h(y))+O(d_M(y)|du^{\epsilon,0}_0|(y))
		\end{align}
		for all $y\in N$ close to $M$. Also, by \cref{gamma.conv} (or from the proof below), the energy densities
		\begin{align*}
			&\mu(u^{\epsilon,0}_0)\weakto\mathcal{H}^m\mrestr M
		\end{align*}
		as $\epsilon\to 0$.
		Hence, when bounding the energy of $u^{\epsilon,w}_0$, we can neglect the complement of $B_{\rho_0}(M)$ and, together with the additional requirement for $u^{\epsilon,0}_0$
		mentioned above, \cref{par.dist} (with a change of variables) gives the convergence
		\begin{align*}
			&\mu(u^{\epsilon,w}_0)\weakto(\Phi_w)_*(J\Phi_w\,\mathcal{H}^m\mrestr M),
		\end{align*}
		where $J\Phi_w$ is the Jacobian of $\Phi_w$. Since the latter equals $J(\Phi_w|_M)$ on $M$, we get
		\begin{align*}
			&\mu(u^{\epsilon,w}_0)\weakto\mathcal{H}^m\mrestr M_w
		\end{align*}
		as $\epsilon\to 0$ (uniformly in $w\in\bp$). Since $J(u^{\epsilon,w}_0)=(\Phi_w)_*J(u^{\epsilon,0}_0)$, we also have
		$J(u^{\epsilon,w}_0)\weakto(\Phi_w)_*\llbracket M\rrbracket =\llbracket M_w\rrbracket $ in the flat norm (uniformly in $w\in\bp$).
		
		For Yang--Mills--Higgs, we observe that the pullback bundle $(\Phi_w^{-1})^*L$ is isomorphic to $L$, with an isomorphism preserving the Hermitian metric:
		one can identify the fiber $L_y$ with the fiber $L_{\Phi_w^{-1}(y)}$ by parallel transport along the curve $t\mapsto\Phi_{tw}^{-1}(y)$.
		Thus, we can let
		\begin{align*}
			&u^{\epsilon,w}_0:=(\Phi_w^{-1})^*u^{\epsilon,0}_0,\quad \nabla^{\epsilon,w}_0:=(\Phi_w^{-1})^*\nabla^{\epsilon,0}_0
		\end{align*}
		and view $(u^{\epsilon,w}_0,\nabla^{\epsilon,w}_0)$ as a couple on $L$, rather than on the pullback bundle. This couple (on $L$) depends continuously on $w$.
		Since $\nabla^{\epsilon,w}_0u^{\epsilon,w}_0$ and the curvature of $\nabla^{\epsilon,w}_0$ are the pullback of the same differential forms for $w=0$ (under the map $\Phi_w^{-1}$), the proof can be concluded as above.
		
		We are left to construct the maps (or couples) for $w=0$. We split the proof into three cases, according to the energy under consideration.
		
		\textbf{Allen--Cahn.} By assumption, we can split $N\setminus M$ into two regions $N_+$ and $N_-$, with common boundary $M$ (having the boundary orientation from $N_-$).
		Also, we can identify
		\begin{align*}
			&B_{\rho_0}(M)\cong M\times B_{\rho_0}^1
		\end{align*}
		(requiring that $B_{\rho_0}(M)\cap N_+$ and $B_{\rho_0}(M)\cap N_-$ are identified with $M\times (0,\rho_0)$ and $M\times (-\rho_0,0)$, respectively).
		
		Recall that the one-dimensional solution to the Allen--Cahn equation is
		\begin{align*}
			&u^\epsilon(v)=\tanh\Big(\frac{v}{\sqrt{2}\epsilon}\Big),
		\end{align*}
		with total energy $E_\epsilon(u^\epsilon)=1$.
		We fix an even cut-off function $\chi\in C^\infty_c(B_{\rho_0}^1)$ such that $0\le\chi\le 1$ and $\chi=1$ on $B_{\rho_0/2}^1$, and define
		\begin{align*}
			&\tilde u^\epsilon(v):=\chi(v)\tanh\Big(\frac{v}{\sqrt{2}\epsilon}\Big)+(1-\chi(v))\sgn(v),
		\end{align*}
		extended to a smooth function on $\R$ (equal to $\sgn(v)$ on the complement of $B_{\rho_0}^1$, where $\sgn(v)$ is the sign function).
		Then we let
		\begin{align*}
			&u^{\epsilon,0}_0(x,v):=\tilde u^\epsilon(v)
		\end{align*}
		for $(x,v)\in M\times B_{\rho_0}^1\cong B_{\rho_0}(M)$.
		
		Note that, for any $r>0$, we have
		\begin{align*}
			&\int_{\R\setminus B_r^1}\Big(\frac{\epsilon|du^\epsilon|^2}{2}+\frac{(1-(u^\epsilon)^2)^2}{4\epsilon}\Big)
			=\int_{\R\setminus B_{r/\epsilon}^1}\Big(\frac{|du^1|^2}{2}+\frac{(1-(u^1)^2)^2}{4}\Big)\to 0
		\end{align*}
		as $\epsilon\to 0$ and that the same conclusion holds for $\tilde u^\epsilon$, since the integrand vanishes on $\R\setminus B_{\rho_0}^1$
		and, on the region $B_{\rho_0}^1\setminus B_{\rho_0/2}^1$, we have
		\begin{align*}
			&|d\tilde u^\epsilon|^2\le 2|du^\epsilon|^2+O(1),\quad 0\le 1-(\tilde u^\epsilon)^2\le 1-(u^\epsilon)^2
		\end{align*}
		(and $\tilde u^\epsilon=u^\epsilon$ on $B_{\rho_0/2}^1$). Thus, for both $u^\epsilon$ and $\tilde u^\epsilon$, the energy density on the real line converges to the Dirac mass $\delta_0$. Since the metric on $M\times B_{r}^1$ is closer and closer to the product metric as $r\to 0$, we deduce that
		\begin{align*}
			&E_\epsilon(u^{\epsilon,0}_0)\to\area.
		\end{align*}
		It is straightforward to check that $J(u^{\epsilon,0}_0)\weakto\llbracket M\rrbracket $ in the flat norm.
		
		\textbf{Ginzburg--Landau.} Recall that we are assuming that $M$ bounds an oriented hypersurface $P$.
		The hypersurface $P$ is two-sided, as $P$ and $N$ are oriented.
		Similarly, since $M$ and $N$ are oriented, the normal bundle $T^\perp M$ is also oriented; $TP|_M$ provides a unit section of it, hence $T^\perp M$ is trivial.
		
		We can then identify $T^\perp M$ with $M\times\R^2$, in such a way that the vector $(x,1)$ belongs to $T_xP$ for all $x\in M$.
		Via the exponential map from $T^\perp M$ to $N$, we also identify $\bar B_{\rho_0}(M)\cong M\times\bar B_{\rho_0}^2$.
		
		Up to modifying $P$ and shrinking $\rho_0$, we can assume that
		\begin{align*}
			&P\cap B_{\rho_0}(M)=M\times([0,\rho_0)\times\{0\})
		\end{align*}
		under this identification. Since the vector field $\nabla d_M$ is tangent to the second factor of $M\times\bar B_{\rho_0}^2$, it can be viewed as a map from $\bar B_{\rho_0}(M)\setminus M$ to $S^1$: specifically, $\nabla d_M(x,v)=\frac{v}{|v|}$
		for $x\in M$ and $v\in\bar B_{\rho_0}^2\setminus\{0\}$.
		
		We extend it to a continuous map $w:N\setminus M\to S^1$:
		it suffices to write $\nabla d_M=e^{i\theta}$ on $\bar B_{\rho_0}(M)\setminus P$, with $0\le\theta\le 2\pi$,
		and extend $\theta$ to a continuous map on $N\setminus P$, such that $\theta\to 0$ and $\theta\to 2\pi$ when approaching $P$ from its two different sides (note that this already holds in $\bar B_{\rho_0}(M)$, since $M$ has the boundary orientation from $P$). We then choose $w:=e^{i\theta}$.
		
		Up to a regularization, we now assume that we have a smooth map $w:N\setminus M\to S^1$ such that $(d_M w)(x,v)=|v|w(x,v)=v$ for $x\in M$ and $v\in\R^2$ small enough.
		Define the Lipschitz map
		\begin{align*}
			&u^{\epsilon,0}_0:=\min\{d_M/\epsilon,1\}w
		\end{align*}
		(defined to be zero on $M$)
		and note that $|u^{\epsilon,0}_0|=1$ outside the neighborhood $B_\epsilon(M)$, which has volume $O(\epsilon^2)$. Also, $d_M w$ is smooth near $M$. Hence,
		\begin{align*}
			&\int_N\frac{(1-|u^{\epsilon,0}_0|^2)^2}{4\epsilon^2}+\int_{B_\epsilon(M)}\frac{|du^{\epsilon,0}_0|^2}{2}\le C
		\end{align*}
		for some constant $C$ independent of $\epsilon\le\rho_0$. Moreover, $|du^{\epsilon,0}_0|=d_M^{-1}(1+O(d_M))$ outside $B_\epsilon(M)$ (as $w(x,v)=\frac{v}{|v|}$ for $v$ small).
		This implies that
		\begin{align*}
			&\int_{\de B_r(M)}\frac{|du^{\epsilon,0}_0|^2}{2}
			=\mathcal{H}^{n-1}(\de B_r(M))\frac{1+O(r)}{2r^2}
			=\frac{\pi}{r}\area+O(1)
		\end{align*}
		for $\epsilon<r<\rho_0$,
		hence $E_\epsilon(u^{\epsilon,0}_0)\to\area$. The convergence of the Jacobians is again immediate, since $J(u^{\epsilon,0}_0)=0$ outside $B_\epsilon(M)$ and $J(u^{\epsilon,0}_0)=\frac{dv^1\wedge dv^2}{\pi\epsilon^2}$ inside. One can then obtain smooth maps by regularization.
		
		\textbf{Yang--Mills--Higgs.} Let $L'\to N$ be the line bundle obtained as follows. Recall that $\pi$ denotes the nearest point projection
		to $M$ (defined on $\bar B_{\rho_0}(M)$) and note that the pullback $\pi^* (T^\perp M)$ is a (real) oriented plane bundle on $\bar B_{\rho_0}(M)$, hence can be viewed as a complex line bundle here. On $\bar B_{\rho_0}(M)\setminus M$ we have a unit section $w$ given by $w(\exp_x(v)):=\frac{v}{|v|}$.
		
		We define $L'$ by gluing $\pi^* (T^\perp M)$ with the trivial line bundle on $N\setminus B_{\rho_0}(M)$,
		identifying $w$ with the section $1$ on $\de B_{\rho_0}(M)$. We extend $w$ to a unit section of $L'$ on $N\setminus M$ by setting $w:=1$ outside $B_{\rho_0}(M)$.
		
		Since $M$ is the zero set of a section, the Euler class $c_1(L')$ is Poincar\'e dual to $\llbracket M\rrbracket $
		(even taking into account orientations). Hence, by assumption, $c_1(L')=c_1(L)$. This implies that $L$ and $L'$ are isomorphic (see \cite[Proposition~4.1]{PPS1})
		and we can just construct the desired couple on $L'$.
		
		
		For any $0<\epsilon\le\rho_0$, we let $(u^\epsilon,d-i\alpha^\epsilon)$ be the (unique) critical couple for the trivial bundle on the plane with degree $1$ at the origin
		and $u^\epsilon/|u^\epsilon|=e^{i\theta}$ (on $\C\setminus\{0\}$).
		
		This couple actually solves the first order vortex equations \cite[III.(1.7)]{JT} and, as seen by integration by parts (and the decay mentioned below),		
		has total energy $E_\epsilon(u^\epsilon,\nabla^\epsilon=d-i\alpha^\epsilon)=1$ (see \cite[III.(1.5)]{JT}).
		We refer the reader to \cite[Theorem~III.2.3]{JT} and the subsequent proof for the construction of this couple when $\epsilon=1$; for the other values of $\epsilon$, we can just take $u^\epsilon:=u^1(\epsilon^{-1}\cdot)$ and $\alpha^\epsilon:=\epsilon^{-1}\alpha^1(\epsilon^{-1}\cdot)$.
		
		For these couples we have $|u^\epsilon|\le 1$ and the following exponential decay (see \cite[Theorem~III.8.1]{JT}; see also the proof of \cite[Corollary~5.4]{PS}):
		\begin{align}\label{exp.decay}
			&|\nabla^\epsilon u^\epsilon|+\frac{1-|u^\epsilon|^2}{\epsilon}+\epsilon|d\alpha^\epsilon|
			\le\frac{C}{\epsilon}e^{-c|z|/\epsilon},
		\end{align}
		where $c,C$ are absolute constants and $|z|$ is the distance from the origin.
		
		We modify this couple in such a way that it has zero energy outside $B_{\rho_0}^2$ (as we did for Allen--Cahn):
		let $\chi:\R^2\to\R$ be a radial smooth function such that $0\le\chi\le 1$, $\chi=1$ on $B_{\rho_0/2}^2$ and $\chi=0$ outside $B_{\rho_0}^2$.
		Then let
		\begin{align*}
			&\tilde u^\epsilon:=\chi u^\epsilon+(1-\chi)e^{i\theta},\quad\tilde\alpha^\epsilon:=\chi\alpha^\epsilon+(1-\chi)\,d\theta
		\end{align*}
		(and $\tilde\nabla^\epsilon:=d-i\tilde\alpha^\epsilon$).
		Outside $B_{\rho_0}^2$ we have $\tilde u^\epsilon=e^{i\theta}$ and $\tilde\nabla^\epsilon=d-i\,d\theta$; hence, the couple $(\tilde u^\epsilon,\tilde\nabla^\epsilon)$ has zero energy here.
		
		Since $u^\epsilon/|u^\epsilon|=e^{i\theta}$, we have $|\tilde u^\epsilon|=\chi|u^\epsilon|+(1-\chi)$, hence $0\le 1-|\tilde u^\epsilon|^2\le 1-|u^\epsilon|^2$. Also, we have
		\begin{align*}
			&\tilde\nabla^\epsilon=\nabla^\epsilon-i(\chi-1)(\alpha^\epsilon-d\theta)=(d-i\,d\theta)-i\chi(\alpha^\epsilon-d\theta),
		\end{align*}
		from which it follows that
		\begin{align*}
			\tilde\nabla^\epsilon\tilde u^\epsilon
			&=(u^\epsilon-e^{i\theta})\,d\chi+\chi\tilde\nabla^\epsilon u^\epsilon+(1-\chi)\tilde\nabla^\epsilon e^{i\theta} \\
			&=(|u^\epsilon|-1)e^{i\theta}\,d\chi+\chi(\nabla^\epsilon u^\epsilon-i(\chi-1)u^\epsilon(\alpha^\epsilon-d\theta)) \\
			&\quad+(1-\chi)(-i\chi e^{i\theta}(\alpha^\epsilon-d\theta))
		\end{align*}
		and
		\begin{align*}
			&d\tilde\alpha^\epsilon
			=d\chi\wedge(\alpha^\epsilon-d\theta)+\chi\,d\alpha^\epsilon.
		\end{align*}
		
		Using the exponential decay \cref{exp.decay} and the inequality
		\begin{align}\label{angular}
			&|\alpha^\epsilon-d\theta|\le|u^\epsilon|^{-1}|\nabla^\epsilon u^\epsilon|,
		\end{align}
		it is easy to deduce that the energy density of $(\tilde u^\epsilon,\tilde\nabla^\epsilon)$ concentrates to the Dirac mass $\delta_0$ as $\epsilon\to 0$.
		
		Using a local orthonormal frame for the normal bundle $T^\perp M$, defined on an open subset $U\subseteq M$, we can identify $\pi^{-1}(U)$
		with $U\times \bar B_{\rho_0}^2$ and $L'$ with $\R^2$ (recall that if $\pi(y)=x$ then $L_y'=T_x^\perp M\cong\R^2$). Here we define
		the desired couple by pullback under the projection $U\times\bar B_{\rho_0}^2\to\bar B_{\rho_0}^2$; namely,
		\begin{align*}
			&u^{\epsilon,0}_0(x,v):=\tilde u^\epsilon(v),\quad \nabla^{\epsilon,0}_0(x,v)=d-i\alpha^{\epsilon,0}_0(x,v):=d-i\tilde\alpha^\epsilon(v)
		\end{align*}
		for $x\in U$ and $v\in\bar B_{\rho_0}^2$ (with $\alpha^{\epsilon,0}_0$ vanishing along vectors parallel to the factor $U$).
		The definition does not depend on the choice of the orthonormal frame, since the couple $(\tilde u^\epsilon,\tilde\alpha^\epsilon)$ is equivariant under rotations (indeed, this holds for the initial couple $(u^\epsilon,\nabla^\epsilon)$ by uniqueness).
		We then let $u^{\epsilon,0}_0:=w$ and $\nabla^{\epsilon,0}_0:=\nabla_w$ outside $\bar B_{\rho_0}(M)$, where $\nabla_w$ is the unique connection (defined on $N\setminus M$) making the unit section $w$ parallel. The resulting couple is smooth (even across $\de B_{\rho_0}(M)$).
		
		Note that the energy of this couple vanishes outside $B_{\rho_0}(M)$. It is easy to check that $E_\epsilon(u^{\epsilon,0}_0,\nabla^{\epsilon,0}_0)\to\area$,
		while the convergence of the Jacobian to $\llbracket M\rrbracket $ can be deduced as follows:
		from \cref{exp.decay} we have $|u^{\epsilon,0}_0|\to 1$ and thus $u^{\epsilon,0}_0\to w$ pointwise (on $N\setminus M$); actually, this also holds in $W^{1,p}(N)$, for all $1\le p<2$,
		since $u^{\epsilon,0}_0=|u^{\epsilon,0}_0|w$ and thus, with respect to a fixed reference connection $\nabla_0$ agreeing with $\nabla_w$ on $N\setminus B_{\rho_0}(M)$,
		\begin{align*}
			&|\nabla_0 (u^{\epsilon,0}_0-w)|
			=|\nabla_0(\chi(u^{\epsilon}-e^{i\theta}))|
			\le|d\chi|\cdot||u^{\epsilon}|-1|+\chi|d|u^{\epsilon}||+\chi||u^{\epsilon}|-1|\cdot|d\theta|.
		\end{align*}
		which is bounded in $L^p(N)$. Similarly, since $|u^{\epsilon}|$ is bounded below by a positive constant on $\R^2\setminus B_\epsilon^2$,
		using \cref{angular} we obtain that
		\begin{align*}
			&\int_{N\setminus B_\epsilon(M)}|\nabla^{\epsilon,0}_0-\nabla_0|^p\le C+C\int_{\R^2\setminus B_\epsilon^2}|\tilde\alpha^\epsilon-d\theta|^p\le C;
		\end{align*}
		using the fact that $\alpha^\epsilon=\epsilon^{-1}\alpha^1(\epsilon^{-1}\cdot)$, we also get
		\begin{align*}
			&\int_{B_\epsilon^2}|\alpha^\epsilon|^p
			=\epsilon^{2-p}\int_{B_1^2}|\alpha^1|^p,
		\end{align*}
		which implies that $\nabla^{\epsilon,0}_0-\nabla_0$ is bounded in $L^p(N)$.
		In particular, $\ang{\nabla_0 u^{\epsilon,0}_0,iu^{\epsilon,0}_0}\to\ang{\nabla_0 w,iw}$ in $L^1(N)$.
		Recalling \cref{jac.ymh}, we arrive at
		\begin{align*}
			&J(u^{\epsilon,0}_0,\nabla^{\epsilon,0}_0)\weakto\frac{1}{2\pi}d\ang{\nabla_0w,iw}+\frac{\omega_0}{2\pi}
		\end{align*}
		(in the flat norm). In a local trivialization, since $|w|=1$, the right-hand side is just $\frac{1}{2\pi}d\ang{dw,iw}$,
		which is ($\frac{1}{\pi}$ times) the usual distributional Jacobian of $w$; using the definition of $w$ near $M$, it is easy to check that the right-hand side equals $\llbracket M\rrbracket $.
	\end{proof}
	
	\section{Deducing the main result (\cref{thm.main}), given \cref{thm.crit}}
	
	Deducing \cref{thm.main} from \cref{thm.crit} is straightforward.
	Fixing $\epsilon>0$ small, for a sequence of times $T_k\to\infty$ we obtain
	solutions $(u^k_t)_{t\in[0,T_k]}$ (or $(u^k_t,\nabla^k_t)_{t\in[0,T_k]}$) to the gradient flow equation
	such that \cref{crit.bds} holds, for a quantity $\delta=\delta(\epsilon)$ such that $\delta(\epsilon)\to 0$ as $\epsilon\to 0$.
	
	In the Allen--Cahn setting, the gradient flow equation
	gives
	\begin{align*}
		&\int_0^{T_k}\int_N|\dot u^k_t|^2
		=(c\epsilon)^{-1}(E_\epsilon(u^k_0)-E_\epsilon(u^k_{T_k})),
	\end{align*}
	so there exists an intermediate time $t_k$ such that
	\begin{align*}
		&\int_N|\dot u^k_{t_k}|^2
		\le\frac{C}{T_k}
	\end{align*}
	for a constant $C$ independent of $k$.
	
	Recalling the gradient flow equation \cref{gf.ac}, setting $u^k:=u^k_{t_k}$ we see that
	\begin{align*}
		&\Delta u^k+\epsilon^{-2}(1-|u^k|^2)u^k\to 0\quad\text{in }L^2.
	\end{align*}
	Since $|u^k|\le 1$, this implies that $(u^k)$ is bounded in $W^{2,2}$. Hence, $u^k\to u^\infty$ in $W^{1,2}$ up to subsequences, and the limit $u^\infty$ is a critical map for $E_\epsilon$ (in particular, $u^\infty$ is smooth). Also, since $J(u^k)\weakto J(u^\infty)$ in the flat norm, we have
	\begin{align*}
		&\fnorm(J(u^\infty)-\llbracket M\rrbracket )
		=\limsup_{k\to\infty}\fnorm(J(u^k)-\llbracket M\rrbracket )
		\le\delta(\epsilon).
	\end{align*}
	Thus, once we set $u_\epsilon:=u^\infty$, we have
	\begin{align*}
		&J(u_\epsilon)\weakto\llbracket M\rrbracket 
	\end{align*}
	in the flat norm, as $\epsilon\to 0$. The $\Gamma$-convergence result (\cref{gamma.conv}) then implies that the energy densities converge (subsequentially) to a measure $\mu$
	with $\mu\ge\mathcal{H}^m\mrestr M$.
	Since
	\begin{align*}
		&E_\epsilon(u_\epsilon)
		=E_\epsilon(u^\infty)
		=\limsup_{k\to\infty}E_\epsilon(u^k)
		\le \limsup_{k\to\infty}E_\epsilon(u_0^k)
		\le\area+\delta(\epsilon),
	\end{align*}
	where the last inequality follows from \cref{prop.constr}, by letting \(\epsilon \to 0\) we deduce  $\mu(N)\le\area$, which forces $\mu=\mathcal{H}^m\mrestr M$.
	The proof for Ginzburg--Landau is identical.
	
	As for Yang--Mills--Higgs, we can find an intermediate time $t_k\in[0,T_k]$ such that, setting $u^k:=u^k_{t_k}$ and $\nabla^k:=\nabla^k_{t_k}$,
	\begin{align}\label{slice.l2.ymh}
	\begin{aligned}
		&-(\nabla^k)^*\nabla^k u^k+\frac{1}{2\epsilon^2}(1-|u^k|^2)u^k\to 0, \\
		&-d^*(\omega_0+d\alpha^k)+\frac{1}{\epsilon^2}\ang{iu^k,\nabla^ku^k}\to 0
	\end{aligned}
	\end{align}
	in $L^2$, where $\omega_0=iF_{\nabla_0}$, for the reference connection $\nabla_0$ (recall that $\nabla^k=\nabla_0-i\alpha^k$).
	Passing to the Coulomb gauge and shifting further the couple by a suitable harmonic $S^1$-valued function
	(see, e.g., \cite[Section~7.1]{PS}), we can replace the couple $(u^k,\alpha^k)$ with a new gauge-equivalent one (denoted in the same way) such that $d^*\alpha^k=0$ and the harmonic part in the Hodge decomposition of $\alpha^k$ is bounded.
	
	From the second part of \cref{slice.l2.ymh} and $\int_N|\nabla^ku^k|^2\le 2\pi E_\epsilon(u^k,\nabla^k)$, we get that
	the Hodge Laplacian $\Delta_H\alpha^k$ is bounded in $L^2$. As the harmonic part of $\alpha^k$ is bounded, standard elliptic estimates imply that $\alpha^k$ is bounded in $W^{2,2}$.
	
	Since $\nabla^ku^k=(\nabla_0-i\alpha^k)u^k$ is bounded in $L^2$, we deduce that $u^k$ is bounded in $W^{1,2}$.
	Hence, $\alpha^k$ and $u^k$ have subsequential weak limits $\alpha^\infty$ in $W^{2,2}$ and $u^\infty$ in $W^{1,2}$, respectively.
	In particular, $\nabla^ku^k\weakto\nabla^\infty u^\infty$ in $L^2$.
	
	Then \cref{slice.l2.ymh} implies that $(u^\infty,\nabla^\infty)$ is a weak solution to the Euler--Lagrange equations for Yang--Mills--Higgs;
	as shown in the appendix of \cite{PS}, the limit couple is smooth and critical for Yang--Mills--Higgs.
	In particular,
	$$\int_N\ang{\nabla^\infty u^\infty,\nabla^\infty(u^k-u^\infty)}=\frac{1}{2\epsilon^2}\int_N\ang{(1-|u^\infty|^2)u^\infty,u^k-u^\infty}\to 0,$$
	while from the first part of \cref{slice.l2.ymh} we get
	$$\int_N\ang{\nabla^k u^k,\nabla^k(u^k-u^\infty)}=\frac{1}{2\epsilon^2}\int_N\ang{(1-|u^k|^2)u^k,u^k-u^\infty}+o(1)\to 0.$$
	Subtracting and using the fact that $\nabla^k-\nabla^\infty\to 0$ in $L^2$, we obtain the strong convergence $\nabla^k u^k\to\nabla^\infty u^\infty$ in $L^2$.
	Hence, from \eqref{jac.ymh}, the Jacobian $J(u^k,\nabla^k)\weakto J(u^\infty,\nabla^\infty)$ in the flat norm. The conclusion follows as before.
	
	\section{Proof of \cref{thm.crit} assuming \cref{thm.core2} and \cref{thm.core}}
	
	In this section we deduce \cref{thm.crit} from \cref{thm.core2} and the rigidity result \cref{thm.core}, following the same strategy of \cite{JSte} and \cite{CJS}.
	
	With abuse of notation, we will write just the map (or section) $u$ even though the arguments cover also the Yang--Mills--Higgs setting
	(where $u$ should be replaced with the couple $(u,\nabla)$). Also, we will assume $\ell\ge 1$; when $\ell=0$, the argument is simpler (see \cref{ell.0} below).
	
	Let $\chi:\bp\to\R$ be a smooth nonnegative function such that $\chi=0$ on $\de\bp$ and $\chi=1$ on the smaller ball $\mz\bp$.
	Starting from the initial map $u^{\epsilon,w}_0$ provided by \cref{prop.constr},
	we let $u^{\epsilon,w}_t$ be the solution to the gradient flow equation at time $\chi(w)t$.
	Thus, for $w\in\mz\bp$, $(u^{\epsilon,w}_t)_{t\in[0,\infty)}$ is the gradient flow of $u^{\epsilon,w}_0$,
	while for $w\in\de\bp$ we have $u^{\epsilon,w}_t=u^{\epsilon,w}_0$.
	
	In the sequel, given a continuous function $F:\bp\to\R^\ell$ and a subset $\Omega\subset\bp$
	such that $F(w)\neq 0$ for all $w\in\de\Omega$ (the topological boundary of $\Omega$ in $\R^\ell$, rather than in $\bp$),
	we denote by
	\begin{align*}
		&\operatorname{deg}(F,\Omega)\in\Z
	\end{align*}
	the topological degree of $F|_\Omega$, with respect to the point $0\in\R^\ell$ in the codomain.
	
	Recalling \cref{P.def}, let
	\begin{align*}
		&f(w,t):=P(J(u^{\epsilon,w}_t))
	\end{align*}
	and note that, for $\epsilon$ small enough, by the second part of \cref{constr.bds} we have
	\begin{align*}
		&|f(w,t)-w|=|P(J(u^{\epsilon,w}_0)-\llbracket M_w\rrbracket )|<|w|
	\end{align*}
	for $w\in\de\bp$, where we used the fact that 
	\(
	P(\llbracket M_w\rrbracket )=w,
	\)
	which follows from \cref{P.graph}. Thus, $f(\cdot,t)$ is homotopic to the identity through maps (from $\bp$ to $\R^\ell$)
	which do not vanish on $\de\bp$; since the identity has degree $1$, we deduce that
	\begin{align*}
		&\operatorname{deg}(f(\cdot,t),\bp)=\operatorname{deg}(\operatorname{id},\bp)=1.
	\end{align*}
	Note that $\fnorm(J(u^{\epsilon,w}_0)-\llbracket M\rrbracket )<\delta_1$ for $\epsilon$ small enough, by \cref{M.w.flat} and \cref{constr.bds}
	(with $\delta_1$ the constant from \cref{thm.core2}).
	Setting
	\begin{align*}
		&A_t(\sigma):=\{w\in\bp:\fnorm(J(u^{\epsilon,w}_t)-\llbracket M\rrbracket )>\sigma\},
	\end{align*}
	we then have $A_0(\delta_1)=\emptyset$, and hence
	\begin{align*}
		&\operatorname{deg}(f(\cdot,0),A_0(\delta_1))=0.
	\end{align*}
	
	We claim that the same holds at time $T$. To this aim, we first check that $f(\cdot,t)\neq 0$ on
	$\bar{A_t(\mz\delta_1)}\setminus A_t(\delta_1)$ for all $t\ge 0$, provided that $\epsilon$ is small enough.
	This follows from \cref{thm.core2}: if we had a sequence $w_j\in\bp$ with $f(w_j,t_j)=0$
	and $w_j\in\bar{A_{t_j}(\mz\delta_1)}\setminus A_{t_j}(\delta_1)$,
	for some $t_j\ge 0$ and $\epsilon=\epsilon_j\to 0$, then
	\begin{align*}
		&\fnorm(J(u^{\epsilon_j,w_j}_{t_j})-\llbracket M\rrbracket )\in[\tmz\delta_1,\delta_1].
	\end{align*}
	By $\Gamma$-convergence (\cref{gamma.conv}), the currents $J(u^{\epsilon_j,w_j}_{t_j})$ would have a subsequential limit $S$
	which is an integral cycle,
	with $\fnorm(S-\llbracket M\rrbracket )\in[\mz\delta_1,\delta_1]$ and
	\begin{align*}
		&\mass(S)\le\liminf_{j\to\infty}E_{\epsilon_j}(u^{\epsilon_j,w_j}_{t_j})
		\le\liminf_{j\to\infty}E_{\epsilon_j}(u^{\epsilon_j,w_j}_0)
		\le\area.
	\end{align*}
	Also, since $f(w_j,t_j)=0$, we would have $P(S)=0$ and hence $\tilde\mass(S)\le\tilde\mass(\llbracket M\rrbracket )$. However, this contradicts the strict minimality of $\llbracket M\rrbracket $ guaranteed by \cref{thm.core2}.
	
	The claim now follows from standard properties of the topological degree.
	Indeed, from the previous remark it follows that $\deg(f(\cdot,t),A_t(\sigma))$ has the same value for all $\sigma\in[\mz\delta_1,\delta_1]$.
	Denoting by $d(t)$ this common value, note that $d(t)$ is a continuous function:
	indeed, for a fixed $t_0\in[0,T]$, we have
	\begin{align*}
		&A_t(\delta_1)\subseteq A_{t_0}(\tfrac{3}{4}\delta_1)\subseteq A_t(\tmz\delta_1)
	\end{align*}
	for all $t$ close to $t_0$ (by uniform continuity of $\fnorm(J(u^{\epsilon,w}_t)-\llbracket M\rrbracket )$ in the couple $(w,t)$).
	By invariance of the degree under homotopies, we then get
	\begin{align*}
		&d(t_0)
		=\operatorname{deg}(f(\cdot,t_0),A_{t_0}(\tfrac{3}{4}\delta_1))
		=\operatorname{deg}(f(\cdot,t),A_{t_0}(\tfrac{3}{4}\delta_1))
		=d(t),
	\end{align*}
	where the last equality comes from the fact that $f(\cdot,t)\neq 0$ on $\bar A_{t_0}(\frac{3}{4}\delta_1)\setminus A_t(\delta_1)$.
	This proves the claim, namely
	\begin{align*}
		&\operatorname{deg}(f(\cdot,T),A_T(\delta_1))=0.
	\end{align*}
	By additivity of the degree, we then have
	\begin{align*}
		&\operatorname{deg}(f(\cdot,T),\bp\setminus A_T(\delta_1))
		=\operatorname{deg}(f(\cdot,T),\bp)-\operatorname{deg}(f(\cdot,T),A_T(\delta_1))
		=1.
	\end{align*}
	This proves that there exists some $w=w(\epsilon,T)$ such that
	\begin{align}
		&P(J(u^{\epsilon,w}_T))=0,\quad \fnorm(J(u^{\epsilon,w}_T)-\llbracket M\rrbracket )\le\delta_1.
	\end{align}
	
	Also, any subsequential limit of a sequence $J(u^{\epsilon_j,w_j}_{T_j})$
	(with $w_j=w(\epsilon_j,T_j)$ as above)
	is an integral cycle $S$ with $P(S)=0$ and $\fnorm(S-\llbracket M\rrbracket )\le\delta_1$.
	Hence, \cref{thm.core2} gives $\mass(S)\ge\area$ but, by $\Gamma$-convergence,
	
	\begin{align}\label{eq.conv}
		&\liminf_{j\to\infty}E_{\epsilon_j}(u^{\epsilon_j,w_j}_{T_j})
		\ge\mass(S)
		\ge\area.
	\end{align}
	In particular, we must have $w_j\to 0$: if we had $|w_j|>r>0$ along a subsequence, \cref{constr.bds} would give
	\begin{align*}
		&E_{\epsilon_j}(u^{\epsilon_j,w_j}_{T_j})
		\le E_{\epsilon_j}(u^{\epsilon_j,w_j}_0)
		\le \mathcal{H}^m(M_{w_j})+\delta(\epsilon_j)
		\le\area-c
	\end{align*}
	for a constant $c>0$ such that $\sup_{w\in\bp\setminus r\bp}\mathcal{H}^m(M_w)<\area-c$, contradicting \cref{eq.conv}.
	
	Since $E_{\epsilon_j}(u^{\epsilon_j,w_j}_{T_j})
	\le E_{\epsilon_j}(u^{\epsilon_j,w_j}_0)$ and
	\begin{align*}
		&\limsup_{j\to\infty} E_{\epsilon_j}(u^{\epsilon_j,w_j}_0)
		\le \area
	\end{align*}
	(again by \cref{constr.bds}), recalling \cref{eq.conv} we deduce that
	\begin{align*}
		&E_{\epsilon_j}(u^{\epsilon_j,w_j}_0)-E_{\epsilon_j}(u^{\epsilon_j,w_j}_{T_j})\to 0.
	\end{align*}
	Since the previous sequences were arbitrary, for $w=w(\epsilon,T)$ as above we deduce
	\begin{align}\label{small.gap}
		&E_\epsilon(u^{\epsilon,w}_0)-E_\epsilon(u^{\epsilon,w}_T)
		\le \delta(\epsilon),\quad |w|\le\delta(\epsilon)
	\end{align}
	for an infinitesimal quantity $\delta(\epsilon)$, possibly different from the one in \cref{constr.bds}.
	The last property implies that $(u^{\epsilon,w}_t)_{t\in[0,T]}$
	is a genuine solution to the gradient flow equation, for $\epsilon$ small enough (as $\chi(w)=1$).
	
	\begin{rmk}\label{ell.0}
		When $\ell=0$ we have $\bp=\{0\}$ and $P\equiv 0$ identically. In this case we can prove a stronger statement: namely,
		\begin{align*}
			&\sup_{t\ge 0}\fnorm(J(u^{\epsilon}_t)-\llbracket M\rrbracket )\to 0
		\end{align*}
		as $\epsilon\to 0$ (we write $u^\epsilon_t$ rather than $u^{\epsilon,0}_t$ for simplicity). If this were not the case, since $\fnorm(J(u^{\epsilon_j}_0)-\llbracket M\rrbracket )\to 0$ (by \cref{constr.bds}),
		we could find $0<\alpha\le\delta_1$ and a sequence $\epsilon_j\to 0$ such that, for some $t_j\ge 0$,
		\begin{align*}
			&\fnorm(J(u^{\epsilon_j}_{t_j})-\llbracket M\rrbracket )=\alpha.
		\end{align*}
		A (subsequential) limit $S$ of $J(u^{\epsilon_j}_{t_j})$ would have $\fnorm(S-\llbracket M\rrbracket )=\alpha$
		and
		\begin{align*}
			&\mass(S)\le\liminf_{j\to\infty}E_{\epsilon_j}(u^{\epsilon_j}_{t_j})\le\liminf_{j\to\infty}E_{\epsilon_j}(u^{\epsilon_j}_0)\le\area,
		\end{align*}
		contradicting \cref{thm.core2}, which asserts that $\llbracket M\rrbracket $ strictly minimizes $\tilde\mass=\mass$ in the flat neighborhood $\bar B_{\delta_1}^\fnorm(\llbracket M\rrbracket )$ (containing $S$).
	\end{rmk}
	
	\cref{thm.crit} now follows from the proposition below, which crucially uses \cref{thm.core} (for $\ell=0$ this additional argument is not needed, as pointed out in the previous remark).
	
	\begin{proposition}
		For $\epsilon$ small enough and any $T\ge0$, taking $w=w(\epsilon,T)$ as above, we have
		\begin{align*}
			&\fnorm(J(u^{\epsilon,w}_t)-\llbracket M\rrbracket )\le\delta'(\epsilon)
		\end{align*}
		for all $t\in[0,T]$, for an infinitesimal quantity $\delta'(\epsilon)$.
	\end{proposition}
	
	\begin{proof}
		Let $\delta_2$ be the constant from \cref{thm.core},
		chosen with $\eta$ to be the minimum of the density lower bounds at times $t\in[1,2]$
		in the asymptotic result \cref{asymp.flow} (for the three energies).
		
		Towards a contradiction, assume that there exist sequences $\epsilon_j\to 0$, $0\le T_j'\le T_j$, and $w_j=w(\epsilon_j,T_j)$ such that, setting $u^j_t:=u^{\epsilon_j,w_j}_{t}$, we have
		\begin{align}\label{alpha}
		&\fnorm(J(u^j_{T_j'})-\llbracket M\rrbracket )\ge\alpha
		\end{align}
		for some positive constant $\alpha$,
		which can be assumed smaller than $\delta_2$.
		
		Since $w_j\to 0$, using also \cref{constr.bds} we get
		\begin{align*}
			&\fnorm(J(u^{\epsilon_j,w_j}_0)-\llbracket M\rrbracket )
			\le\fnorm(J(u^{\epsilon_j,w_j}_0)-\llbracket M_{w_j}\rrbracket )+\fnorm(\llbracket M_{w_j}\rrbracket -\llbracket M\rrbracket )
			\to 0;
		\end{align*}
		thus, by continuity 
		we can actually assume that
		\begin{align}\label{alpha.bis}
			&\fnorm(J(u^j_{T_j'})-\llbracket M\rrbracket )=\alpha\le\delta_2.
		\end{align}
		
		Note that, by \cref{small.gap} and the estimates preceding it, we have
		\begin{align}\label{const.energy}
			&\lim_{j\to\infty}E_{\epsilon_j}(u^j_{t_j})
			=\area
		\end{align}
		for any sequence of times $t_j\in[0,T_j]$.
		
		We first show that $T_j'\to\infty$ as $j\to\infty$ and, in particular, $T_j'\ge 2$ eventually.
		Assume by contradiction that $\sup_j T_j'<\infty$ along a subsequence (not relabeled).
		Note that the energy densities $\mu(u^j_0)$ converge to $\mathcal{H}^m\mrestr M$ (by the same argument used in \cref{gamma.rmk}, since $\llbracket M\rrbracket $ is the limit of $J(u^j_t)$).
		
		Further, let $S$ be the limit of $J(u^j_{T_j'})$ and $\mu$ the limit of $\mu(u^j_{T_j'})$, up to subsequences; \cref{gamma.conv} gives the bound $|S|\le\mu$.
		We claim that
		\begin{align}\label{const.energy.meas}
			&\mu=\mathcal{H}^m\mrestr M,
		\end{align}
		which implies that $S$ is supported in $M$ and thus, by the constancy theorem and the integrality of $S$,
		$S$ is a union of connected components of $M$, up to changes in orientation. But \cref{alpha.bis} gives $\fnorm(S-\llbracket M\rrbracket )\le\delta_2$, which then implies $S=\llbracket M\rrbracket $ once we assume (without loss of generality) that $\delta_2$ is small enough. This contradicts the fact that
		$\fnorm(S-\llbracket M\rrbracket )=\alpha>0$.
		
		In order to check \cref{const.energy.meas}, for the Allen--Cahn energy we observe that
		\begin{align*}
			&\de_t\Big(\mz|du^j_t|^2+\frac{1}{4\epsilon_j^2}(1-|u^j_t|^2)^2\Big)
			=-|\dot u^j_t|^2-d^*(\dot u^j_t\,du^j_t),
		\end{align*}
		by direct computation using the gradient flow equation for $(u^j_t)$. Integrating over $[0,T_j']\times N$ against any test function $\phi$ (constant in time),
		we get
		\begin{align}\label{const.energy.test}
			&(c\epsilon_j)^{-1}\int_N\phi\,d(\mu^j_{T_j'}-\mu^j_0)
			=-\int_0^{T_j'}\int_N|\dot u^j_t|^2\phi-\int_0^{T_j'}\int_N\ang{\dot u^j_t\,d\phi,du^j_t}
		\end{align}
		(for $c=\frac{3}{2\sqrt{2}}$), where we set $\mu^j_t:=\mu(u^j_t)$.
		Using this identity with $\phi=1$ (or using directly the definition of gradient flow), we see that
		\begin{align*}
			&\int_0^{T_j'}\int_N|\dot u_t|^2
			=(c\epsilon_j)^{-1}(E_{\epsilon_j}(u^j_0)-E_{\epsilon_j}(u^j_{T_j'}))
			=o(\epsilon_j^{-1}).
		\end{align*}
		Since $\sup_j T_j'$ is finite, we have
		\begin{align*}
			&\int_0^{T_j'}\frac{|du^j_t|^2}{2}
			\le (c\epsilon_j)^{-1}\int_0^{T_j'}E_{\epsilon_j}(u^j_t)
			\le (c\epsilon_j)^{-1}T_j'E_{\epsilon_j}(u^j_0)
			=O(\epsilon_j^{-1}).
		\end{align*}
		Thus, by Cauchy--Schwarz, the right-hand side of \cref{const.energy.test} is infinitesimal with respect to $\epsilon_j^{-1}$ for any fixed $\phi$;
		this proves the claim, since the limit of $\mu^j_{T_j'}$ is then equal to the limit of $\mu^j_0$, which is $\mathcal{H}^m\mrestr M$.
		For Ginzburg--Landau the proof is the same, replacing the factor $(c\epsilon_j)^{-1}$ with $\pi\log(\epsilon_j^{-1})$
		(and $\dot u^j_t\,du^j_t$ with $\ang{\dot u^j_t, du^j_t}$).
		
		Finally, for Yang--Mills--Higgs direct computation gives
		\begin{align*}
			&\de_t\Big(|\nabla_t u^j_t|^2+\frac{1}{4\epsilon_j^2}(1-|u^j_t|^2)^2+\epsilon_j^2|\omega^j_t|^2\Big) \\
			&=-2(|\dot u_t|^2+\epsilon_j^2|\dot\alpha^j_t|^2)
			-2d^*(\ang{\nabla u^j_t,\dot u^j_t}+\epsilon^2\omega^j_t(\cdot,\dot\alpha^j_t)),
		\end{align*}
		where $\dot\alpha^j_t$ is identified with a vector field in the last term. The claim follows in an analogous way.
		
		Since we now know that $T_j'\to\infty$, and in particular $T_j'\ge 2$ eventually, for $j$ large enough we can define $v^j_t:=u^j_{T_j'-2+t}$ (for $t\in[0,2]$).
		We can now apply \cref{asymp.flow}: in the Allen--Cahn case, up to a subsequence, the energy density $\mu(v^j_t)$ converges to a limit measure $\mu_t$ for all $t\in[0,2]$,
		and the family $(\mu_t)_{t\in[0,2]}$ is an $(n-1)$-dimensional Brakke flow.
		Since
		\begin{align*}
			&E_{\epsilon_j}(v^j_0)-E_{\epsilon_j}(v^j_2)\to 0,
		\end{align*}
		the same argument used above implies that the limit $\mu_t$ does not depend on $t$;
		we call $\mu$ the common limit. It is the weight of a rectifiable varifold of mass $\mathcal{H}^m(M)$, and its $(n-1)$-density is bounded below by a positive constant $\eta$,
		since this is true for the measure $\mu_t$, for a.e.\ $t\in[1,2]$ (actually, $\eta=1$ for Allen--Cahn).
		
		By Brakke's inequality (tested with the constant function $1$), we deduce that
		this varifold is stationary.
		
		By the $\Gamma$-convergence result (\cref{gamma.conv}), the (subsequential) limit of $J(v^j_2)=J(u^j_{T_j'})$ is an integral cycle $S$ with $|S|\le\mu_2=\mu$.
		By \cref{alpha.bis}, 
		\begin{align*}
			&\fnorm(S-\llbracket M\rrbracket )=\alpha\in(0,\delta_2].
		\end{align*}
		Also, $\mu(N)=\area$. \cref{thm.core} then applies and gives the desired contradiction.
		
		The proof is identical in the Yang--Mills--Higgs setting. As for Ginzburg--Landau, \cref{asymp.flow} gives
		\begin{align*}
			&\mu
			=\mu_t
			=\mz|\alpha_t|^2\,\mathcal{H}^n+\nu_t
		\end{align*}
		for $t\in[0,2]$, where $\alpha_t$ solves the heat equation $\de_t\alpha_t=-dd^*\alpha_t$ and $(\nu_t)_{t\in[0,2]}$ is an $(n-2)$-dimensional Brakke flow.
		Since the first term is the absolutely continuous part of $\mu_t=\mu$ (for a.e.\ $t$),
		it must be constant in time. Thus, the same holds for $\nu_t=\nu$.
		Again, a subsequential limit $S$ of $J(u^j_{T_j'})$ has $|S|\le\mu$ and,
		since it is singular with respect to $\mathcal{H}^N$, we actually get $|S|\le\nu$.
		Since the $(n-2)$-density of $\nu$ is at least $\eta$, we reach again a contradiction to \cref{thm.core}.
	\end{proof}

	\begin{rmk}
		Without the second part of the proof (which requires \cref{thm.core}),
		using the argument used in the first part we can still show that, for all $T\le\tilde T(\epsilon)$,
		we have $\fnorm(J(u^{\epsilon,w}_t)-\llbracket M\rrbracket )\le\tilde\delta(\epsilon)$ for all $t\in[0,T]$
		($w$ depends on $\epsilon$ and $T$),
		with $\tilde\delta(\epsilon)\to 0$ and $\tilde T(\epsilon)\to\infty$. However, this is not enough to deduce \cref{thm.main}:
		in order to find critical maps we need arbitrarily large intervals $[0,T]$ for a fixed $\epsilon$.
	\end{rmk}

	\section{Proof of \cref{thm.core2}: strict local minimality of $M$ for the perturbed mass functional}
	
	This section is devoted to the proof of \cref{thm.core2}. The proof follows the (by now) well-known observation that, combining a compactness-and-contradiction argument with the regularity theory for mass-minimizing currents, one can improve infinitesimal minimality (i.e., positive second variation) to a local one, in a suitable weak topology; we refer, for instance, to \cite{DM,IM, White} for applications to mass-minimizing currents, and to \cite{CL} for applications to isoperimetric-type inequalities. Although the argument is quite standard, we report it here for the sake of completeness, and also because our setting is slightly different from the works cited above (cf.\ also \cite{Mesa}, where a very similar statement is proved in dimension one).
	
	In the following,  we let  $\rho_0$ be the constant fixed in \cref{sec.prel}.
	
	\begin{lemmaen}\label{cut}
		For any $0<\rho\le\rho_0/2$ there exists $\delta(M,\rho)>0$ such that, if $S$ is an $m$-cycle with $\fnorm(S-\llbracket M\rrbracket )\le\delta$ and
		\begin{align*}
			&\tilde\mass(S)
			\le\tilde\mass(\llbracket M\rrbracket )
			=\mathcal{H}^m(M),
		\end{align*}
		then there exists $S'$ supported in the tubular neighborhood $\bar B_\rho(M)$, with
		\begin{align*}
			&\tilde\mass(S')\le\tilde\mass(S).
		\end{align*}
		Moreover, if $S\neq \llbracket M\rrbracket $ then also $S'\neq \llbracket M\rrbracket $.	
	\end{lemmaen}
	
	In the proof, we will use the following simple well-known facts.
	
	\begin{lemmaen}\label{modmass}
		There exists a small constant $c(N)>0$ with the following property.
		Given an integral $k$-cycle $S$ on $N$, with $1\le k\le n-1$, if $\mass(S)\le c(N)$ then $S=\de R$, for an integral $(k+1)$-current $R$ with
		\begin{align*}
			&\mass(R)\le C(N)\mass(S)^{(k+1)/k}.
		\end{align*}
	\end{lemmaen}
	
	\begin{proof}
		%
		We can view $N$ as an embedded submanifold of some Euclidean space $\R^L$ (for the present purposes, we do not need an isometric embedding).
		Identifying $S$ with a cycle $\tilde S$ in $\R^L$, by \cite[Theorem~30.1]{Simon} and its proof we can find an integral $(k+1)$-current $\tilde R$ such that $\tilde S=\de \tilde R$ and
		\begin{align*}
			&\mass(\tilde R)
			\le C(k)\mass(\tilde S)^{(k+1)/k}
			\le C(N)\mass(S)^{(k+1)/k},
		\end{align*}
		with $\tilde R$ supported in the $r$-neighborhood of $\spt(\tilde S)\subseteq N$, for a distance
		\begin{align*}
			&r=C(k)\mass(\tilde S)^{1/k}
			\le C(N)\mass(S)^{1/k}.
		\end{align*}
		If $\mass(S)$ (and thus $r$) is small enough, we can project $\tilde R$ onto $N$ using the nearest point projection, which is Lipschitz near $N$, obtaining the desired $R$.
	\end{proof}
	
	\begin{lemmaen}\label{no.canc}
		Given a sequence of rectifiable $k$-currents $S_j\weakto S_\infty$ converging in the weak topology, if $\mass(S_j)\to\mass(S_\infty)$ then the weights $|S_j|\weakto|S_\infty|$, as Radon measures on $N$.
	\end{lemmaen}
	
	
	\begin{proof}
		Up to subsequences, assume that $|S_j|\weakto\mu$ for a Radon measure $\mu$. Then, for any compact set $K\subseteq N$ and any $r>0$,
		lower semicontinuity of the mass gives
		\begin{align*}
			&|S_\infty|(N\setminus\bar B_r(K))
			\le\liminf_{j\to\infty}|S_j|(N\setminus\bar B_r(K)),
		\end{align*}
		and thus
		\begin{align*}
			\mu(B_r(K))+|S_\infty|(N\setminus\bar B_r(K))
			&\le\liminf_{j\to\infty}|S_j|(B_r(K))+\liminf_{j\to\infty}|S_j|(N\setminus\bar B_r(K)) \\
			&\le\lim_{j\to\infty}|S_j|(N).
		\end{align*}
		By assumption, the latter equals $|S_\infty|(N)$.
		This implies $\mu(B_r(K))\le|S_\infty|(\bar B_r(K))$ and, letting $r\to 0$, we deduce that $\mu\le|S_\infty|$.
		Since $\mu(N)=\lim_{j\to\infty}|S_j|(N)=|S_\infty|(N)$, we must then have $\mu=|S_\infty|$.
	\end{proof}
	
	\begin{proof}[Proof of \cref{cut}]
		For any $r>0$ we denote
		\begin{align*}
			&S_r:=\de(S\mrestr B_r(M)),
		\end{align*}
		which is the $r$-slice of $S$ by the Lipschitz function $d_M$ (see \cite[Section~28]{Simon}).
		The current $S_r$ is an integral $(m-1)$-cycle for a.e.\ $r>0$, and its mass satisfies
		\begin{align*}
			&\int_0^\infty\mass(S_r)\,dr
			\le\mass(S)
		\end{align*}
		(as $d_M$ is $1$-Lipschitz on $N$). Next, we set
		\begin{align*}
			&S_{<r}:=S\mrestr B_r(M),\quad S_{\ge r}:=S\mrestr (N\setminus B_r(M)),
		\end{align*}
		and define the decreasing function
		\begin{align*}
			&f:[0,\rho]\to\R,\quad f(r):=\mass(S_{\ge r}),
		\end{align*}
		so that the coarea formula for slices gives 
		\begin{align}\label{der.f}
			&\mass(S_r)
			\le -f'(r)
		\end{align}
		for a.e.\ $r$. We can assume that $S$ is not supported in $\bar B_\rho(M)$, since otherwise we can just take $S':=S$. In particular, $f\ge f(\rho)>0$.
		
		Given $r\in(0,\rho)$ such that $S_r$ is an integral cycle and \cref{der.f} holds, if $\mass(S_r)$ is small enough then, using \cref{modmass}, we can find an integral $m$-current $R_r$ such that
		\begin{align}\label{S}
			&\de R_r=S_r,\quad
			\mass(R_r)
			\le C(N)\mass(S_r)^{m/(m-1)}
		\end{align}
		for $m\ge 2$; note that for $m=1$ we have $S_r=0$ once we assume that $\mass(S_r)<1$ (since $S_r$ is an integral $0$-current), and we can take $R_r:=0$ in the sequel.
		Also, as the proof of the lemma shows, $R_r$ is supported near $\spt(S_r)\subseteq B_{\rho_0/2}(M)$,
		hence in $\bar B_{\rho_0}(M)$ (for $\mass(S_r)$ small enough). Thus, applying the nearest point projection $\bar B_{\rho_0}(M)\to\bar B_r(M)$, which is $C(M)$-Lipschitz, we can assume that $\spt(R_r)\subseteq\bar B_r(M)$.
		
		Defining $S':=S_{<r}-R_r$, note that $\de S'=S_r-\de R_r=0$ and that
		\begin{align}\label{mass.gain}
			&\mass(S')
			\le\mass(S)-\mass(S_{\ge r})+\mass(R_r)
			=\mass(S)-f(r)+\mass(R_r).
		\end{align}
		
		Also, we claim that
		\begin{align}\label{P.loss}
		\begin{aligned}
			|P(S')|^2-|P(S)|^2
			&\le |P(S+S')|\,|P(S-S')| \\
			&\le C(M)(\mass(S)+\mass(S'))(f(r)+\mass(R_r))^{(m+1)/m}
		\end{aligned}
		\end{align}
		provided that $f(r)+\mass(R_r)$ is small enough. Indeed, in this case, using \cref{modmass} we have
		\begin{align*}
			&S-S'=S_{\ge r}+R_r=\de Q_r
		\end{align*}
		for some $(m+1)$-current $Q_r$ of mass bounded by $C(N)(f(r)+\mass(R_r))^{(m+1)/m}$, and using Stokes we obtain
		\begin{align*}
			&|P(S-S')|
			\le\sum_{j=1}^\ell\lvert\ang{\de Q_r,\omega_j}\rvert
			\le C(M)\mass(Q_r),
		\end{align*}
		which gives the above bound. To sum up, assume that
		\begin{align}\label{small.f}
			&f(r)<c,\quad f'(r)>-cf(r)^{(m-1)/m},
		\end{align}
		for a suitable small constant $c$. Then, for $m>1$, \cref{der.f} and \cref{S} give
		\begin{align*}
			&\mass(S_r)
			\le cf(r)^{(m-1)/m}
			\le c, \quad
			3\mass(R_r)
			\le C(N)c^{m/(m-1)}f(r)
			<f(r),
		\end{align*}
		and all the previous bounds hold true. In particular, \cref{mass.gain} gives
		\begin{align*}
			&\mass(S')
			\le\mass(S)
			\le\mathcal{H}^m(M),
		\end{align*}
		while 
		summing \cref{mass.gain} and \cref{P.loss} we obtain
		\begin{align*}
			\tilde\mass(S')
			&\le\tilde\mass(S)-f(r)+\mass(R_r)
			+C(M)(f(r)+\mass(R_r))^{(m+1)/m} \\
			&\le\tilde\mass(S)
			-f(r)+\mass(R_r)
			+\mz(f(r)+\mass(R_r)) \\
			&<\tilde\mass(S)
		\end{align*}
		for $c$ small enough,
		where we used again the bound $3\mass(R_r)<f(r)$ in the last inequality. This also holds for $m=1$, since
		\begin{align*}
			&\tilde\mass(S')
			\le\tilde\mass(S)-f(r)
			+C(M) f(r)^{2}
			\le\tilde\mass(S)
			-\mz f(r)
			<\tilde\mass(S)
		\end{align*}
		as $0<f(r)<c$.
		Also, we cannot have $S'=\llbracket M\rrbracket $, since $\tilde\mass(S')<\tilde\mass(S)\le\tilde\mass(\llbracket M\rrbracket )$.
		Note that $c=c(M)$.
		
		In order to conclude, we need to check that conditions \cref{small.f} hold on a set of radii $r$ of positive measure, provided that $S$ is close enough to $\llbracket M\rrbracket $ in the flat norm.
		We first show that $f(r)<c$ for all $r\in(\rho/2,\rho)$. Indeed, assume by contradiction that we have
		\begin{align*}
			&|S_j|(N\setminus B_{\rho/2}(M))\ge c
		\end{align*}
		for a sequence of $m$-cycles such that $\mass(S_j)\le\mathcal{H}^m(M)$ and $S_j\weakto\llbracket M\rrbracket $. Since $\mathcal{H}^m(M)\le\liminf_{j\to\infty}\mass(S_j)$ by lower semicontinuity of the mass, we have
		\begin{align*}
			&\mathcal{H}^m(M)=\lim_{j\to\infty}\mass(S_j).
		\end{align*}
		Hence, \cref{no.canc} applies (with $S_\infty=\llbracket M\rrbracket $) and gives $|S_j|\weakto\mathcal{H}^m\mrestr M$. In particular,
		\begin{align*}
			&0
			=\mathcal{H}^m\mrestr M(N\setminus B_{\rho/2}(M))
			\ge\limsup_{j\to\infty}|S_j|(N\setminus B_{\rho/2})
			\ge c,
		\end{align*}
		which is impossible.
		Finally, assume by contradiction that $f'(r)\le -cf(r)^{(m-1)/m}$ for a.e.\ $r\in(\rho/2,\rho)$. Recalling that $f(r)>0$ for all $r\le\rho$, the differential inequality becomes
		\begin{align*}
			&(f^{1/m})'
			\le -\frac{c}{m}
		\end{align*}
		a.e., which gives $0\le f(\rho)^{1/m}\le f(\rho/2)^{1/m}-\frac{c\rho}{2m}$ since $f$ is decreasing. This implies $f(\rho/2)\ge(\frac{c\rho}{2m})^m$, which is again false if $S$ is close enough to $\llbracket M\rrbracket $ in the flat norm.
	\end{proof}
	
	We will now show  that $\llbracket M\rrbracket $ is a strict local minimum of $\tilde\mass$ with respect to the flat metric.
	Fix a radius $0<\rho\le{\rho_0}/{2}$. Since a sequence of integral $m$-cycles with bounded mass has a converging subsequence in the weak topology, with respect to which the map $P$ is continuous, we can find $S$ which minimizes the functional $\tilde\mass$, among integral $m$-cycles in the closed flat neighborhood $\bar B_\delta^{\mathcal{F}}(\llbracket M\rrbracket )$, with
	\begin{align*}
		&\tilde\mass(S)\le\tilde\mass(\llbracket M\rrbracket )=\mathcal{H}^m(M).
	\end{align*}
	We let $\delta_2:=\delta(M,\rho)$, with $\delta$ given by \cref{cut}: we can then assume that $S$ is supported in $\bar B_\rho(M)$.
	Assuming that $S\neq\llbracket M\rrbracket $, we will reach a contradiction once we select $\rho$ small enough.
	
	We claim that there exists a minimizer $\bar S$ (for $\tilde\mass$) among all integral $m$-cycles of the form $\llbracket M\rrbracket +\de R$, with $R$ an $(m+1)$-current supported in $\bar B_\rho(M)$.
	Indeed, for a minimizing sequence $(S_j)$, let $\bar S$ be the weak limit of $S_j$ (up to subsequences).
	By the deformation theorem (see, e.g., \cite[Corollary~29.3]{Simon}), viewing $N$ as an embedded submanifold in $\R^L$, we can write $S_j=P+\de Q_j$ for a sequence $(Q_j)$ of integral currents of bounded mass, supported near $\bar B_\rho(M)$, and a constant (up to subsequences) polyhedral cycle $P$. Assuming $Q_j\weakto\bar Q$, we then get $\bar S=P+\de\bar Q$. Thus,
	\begin{align*}
		&\bar S-\llbracket M\rrbracket 
		=(\bar S-S_j)+(S_j-\llbracket M\rrbracket )
		=\de(\bar Q-Q_j)+(S_j-\llbracket M\rrbracket )
	\end{align*}
	is the boundary of a current supported near $\bar B_\rho(M)$. By projecting this current first onto $N$ and then onto $\bar B_\rho(M)$, it follows that $\bar S$ is the desired minimizer.
	
	Note that, since $S$ is a competitor, we have
	\begin{align*}
		&\tilde\mass(\bar S)
		\le\tilde\mass(S)
		\le\mathcal{H}^m(M).
	\end{align*}
	We can assume that $\bar S\neq\llbracket M\rrbracket $ (if $\bar S=\llbracket M\rrbracket$ then $\tilde\mass(\bar S)=\tilde\mass(S)$, and we can replace $\bar S$ with $S\neq\llbracket M\rrbracket$).	
We now show that \(\bar S\) satisfies a suitable minimality condition.

\begin{lemmaen}
There exists a constant \(C=C(M)\) such that, for all  \(0<r<\rho_0/2\)  and all integral $m$-cycles \(X\) with \(\spt(X)\subseteq B_{\rho_0}(M)\) and  \(\diam(\spt (X))\le r\),
\[
\mass(\bar S)\le (1+Cr)\mass(\bar S+X)+Cr\mass(X).
\]
\end{lemmaen}

\begin{proof}
Let \(X\) be as in the statement  and let \(R\) be the cone over \(X\) (taken from a point in $\spt(X)$). Note that \(\partial R=X\) and
\[
\mass(R)\le C(N)r \mass(X).
\]
We can assume that the support of $X$ intersects $\bar B_\rho(M)$, since otherwise the inequality is trivial.
Define \(T:=(\pi_{\rho})_{*}(\bar S+X)=\bar S+(\pi_{\rho})_* X\), where \(\pi_{\rho}:B_{\rho_0}(M) \to \bar B_{\rho}(M)\) is the nearest point projection. Since \(\pi_{\rho} \) is $(1+C(M)r)$-Lipschitz on $\bar B_{\rho+r}(M)$, we can bound
\[
\mass(T)\le (1+Cr)\mass(\bar S+X),\quad \mass((\pi_{\rho})_* R)\le (1+Cr)\mass(R)\le Cr\mass(X).
\]
Using Stokes as in the proof of \cref{P.loss}, we obtain
$$ |P(T)-P(\bar S)|=|P((\pi_\rho)_*X)|=|P(\de(\pi_\rho)_*R)|\le C\mass((\pi_\rho)_*R)\le Cr\mass(X). $$
Note that $|P(\bar S)|^2\le\tilde\mass(\bar S)\le\mathcal{H}^m(M)$. Also, we can assume that $\mass(\bar S+X)\le 2\mass(\bar S)$,
since otherwise the statement is trivial (provided that $1+C\rho\le 2$). In particular, we have $|P(\bar S+X)|\le C\mass(\bar S+X)\le C\mass(\bar S)$; hence,
$$ |P(\bar S)|+|P(T)|\le C. $$
By minimality of $\bar S$, we have $\tilde\mass(\bar S)\le\tilde\mass(T)$, and it follows that
\begin{align*}
\mass(\bar S)&\le \mass(T)+C|P(T)-P(\bar S)|\le(1+Cr)\mass(\bar S+X)+Cr\mass(X). \qedhere
\end{align*}
\end{proof}

	We now show that such $\bar S\neq\llbracket M\rrbracket $ cannot exist for $\rho$ small enough. Note that $\pi_*\bar S=\llbracket M\rrbracket $ (with $\pi:\bar B_\rho(M)\to M$ the nearest point projection), since $\pi_*\bar S$ is an $m$-cycle homologous to $\llbracket M\rrbracket $ in $M$, and since any $(m+1)$-current in $M$ vanishes. Hence,
	\begin{align*}
		&\flat(\bar S-\llbracket M\rrbracket )
		=\flat(\bar S-\pi_*\bar S)
		\le C(M)\rho\mass(\bar S)
		\le C(M)\rho,
	\end{align*}
	as points $y\in\bar B_\rho(M)$ can be connected to $\pi(y)$ with a geodesic segment of length at most $\rho$.
	
	Using the notation $\bar S=\bar S^\rho$ to emphasize the dependence on $\rho$, it follows that $\bar S^\rho$ is arbitrarily close to $\llbracket M\rrbracket $ in the flat norm.
	From the lower semicontinuity of the mass and the inequality $\mass(\bar S^\rho)\le\mathcal{H}^m(M)$, we also have
	\begin{align*}
		&\lim_{\rho\to 0}\mass(\bar S^\rho)
		=\mathcal{H}^m(M).
	\end{align*}
	From this and \cref{no.canc} it follows that the weight measure $|\bar S^\rho|$ converges to $\mathcal{H}^m\mrestr M$. In particular the cylindrical excess of $\bar S^\rho$, as defined in \cite{DS}, converges to zero in suitable charts adapted to $M$.
	
	By the main result of \cite{DS}, namely \cite[Theorem~6.1]{DS}\footnote{The previous lemma implies that $\mass(\bar S)\le(1+Cr)\mass(\bar S+X)+Cr\mass(\bar S\mrestr\operatorname{spt}(X)+X)$, for a possibly different $C$. In \cite{DS} another definition of almost minimality is used, where the factor $(1+Cr)$ does not appear; however, the proof from \cite{DS} can be adapted with trivial changes.}  (see also \cite{B,SS}), which applies thanks to the previous lemma, this then implies that $\bar S^\rho$ is eventually a graph over $M$, converging to $M$ in the $C^{1,\alpha}$ topology for some $\alpha>0$.
%
	More precisely, $\bar S^\rho=\llbracket M_{\varphi^\rho}\rrbracket $ for a nontrivial section $\varphi^\rho\to 0$ in $C^{1,\alpha}(M,T^\perp M)$
	(using the notation of \cref{sec.prel}).
	For $\rho$ small enough, the inequality $\tilde\mass(\bar S^\rho)\le\area$ contradicts \cref{second.var.bd.cons}. This completes the proof of \cref{thm.core2}.
	
	\section{A rigidity result for non-degenerate minimal submanifolds among varifolds}
	\subsection{Proof of \cref{thm.core.main}}
	Let $M^m$ be a (smooth, embedded) non-degenerate minimal submanifold, in a closed ambient $(N^n,g)$;
	we do not need $M$ or $N$ to be orientable for this theorem to hold.
	
	For a fixed small $0<\rho<\operatorname{inj}(N)$, let $U:=B_\rho(M)$ be the tubular neighborhood of $M$ with radius $\rho$, and let $\pi:U\to M$ be the nearest point projection.
	Let us recall the statement of \cref{thm.core.main}, for the reader's convenience.
	
	\begin{thm*}
		Let \(M\) and \(N\) be as above. Given $\eta,\Lambda>0$, there exists $\delta(M,N,\eta,\lambda)\in(0,\rho)$ with the following property:
		if a stationary rectifiable $m$-varifold $V$ is such that
		\begin{itemize}
		\item $d_H(\spt|V|,M)\le\delta$ (with $d_H$ the Hausdorff distance),
		\item $\mathbb{M}(V)\le\Lambda$ (with $\mathbb{M}(V)$ the total mass of $V$),
		\item $\Theta^m(|V|,y)\ge\eta$ for all $y\in\spt|V|$,
		\item $\spt|V|\cap\pi^{-1}(x)$ is nonempty for all $x\in M$, and consists of a single point for $x\in M\setminus E$, for a set $E\subseteq M$ with $\mathcal{H}^m(E)\le\delta$,
		\end{itemize}
		then $\spt|V|=M$.
	\end{thm*}
	
	Assuming \cref{thm.core.main} for a moment, let us show how to deduce \cref{thm.core}.
	
	\begin{proof}[Proof of \cref{thm.core}]
		Assume by contradiction that the sequences $(V_j)$ and $(S_j)$ satisfy all the assumptions (with $V=V_j$ and $S=S_j$), for a sequence $\delta_j\to 0$
		(in place of $\delta_2$). Then lower semicontinuity of the mass gives
		\begin{align*}
			&\mathcal{H}^m(M)
			\le\liminf_{j\to\infty}\mass(S_j)
			\le\limsup_{j\to\infty}\mass(V_j)
			\le\mathcal{H}^m(M).
		\end{align*}
		Hence, $\mass(S_j)$ and $\mass(V_j)$ both converge to the area of $M$. Assuming $V_j\weakto V_\infty$ up to subsequences,
		using 
		\cref{no.canc} we get
		\begin{align*}
			&\mathcal{H}^m\mrestr M
			=\lim_{j\to\infty}|S_j|
			\le\lim_{j\to\infty}|V_j|
			=|V_\infty|.
		\end{align*}
		Since $\mathcal{H}^m\mrestr M$ and $|V_\infty|$ have the same total mass, this forces
		\begin{align*}
			&|V_\infty|=\mathcal{H}^m\mrestr M.
		\end{align*}
		The stationary varifold $V_\infty$ is then supported in $M$ and the previous equality, together with the constancy theorem (or the fact that $V_\infty$ is rectifiable), implies that $V_\infty=M$ (with multiplicity one). Also, by Hausdorff convergence  of the support (which follows from the monotonicity formula and the lower density bound), we get
		\begin{align}\label{conv.spt}
			&d_j:=\max\{\operatorname{dist}(y,M)\mid y\in\spt|V_j|\}\to 0.
		\end{align}
		
		From \cref{conv.spt} and $|S_j|\le|V_j|$, eventually we have $\spt|S_j|\subseteq U=B_\rho(M)$. By the constancy theorem (and integrality), $\pi_*S_j$ is a constant (integer) multiple of $\llbracket M\rrbracket $, on each connected component of $M$;
		since this multiple is bounded above and $\pi_*S_j\weakto\llbracket M\rrbracket $, we deduce that
		\begin{align*}
			&\pi_*S_j=\llbracket M\rrbracket 
		\end{align*}
		eventually.
		Note that $S_j$ is not supported in $M$ (eventually), since this would imply $S_j=\pi_*S_j=\llbracket M\rrbracket $, contradicting the assumption that $S_j\neq\llbracket M\rrbracket$.
		
		Thus, $\spt|S_j|\cap\pi^{-1}(x)\neq\emptyset$ for a.e.\ $x\in M$ (eventually), and hence for all $x\in M$ since the support is compact.
		The same must then hold for $|V_j|$, as $\spt|V_j|\supseteq\spt|S_j|$.
		Since $S_j$ is integral and the density of $V_j$ is upper semicontinuous, we also deduce that each fiber $\spt|V_j|\cap\pi^{-1}(x)$ contains at least one point of density at least $1$.
		
		By the area formula, denoting by $E_j\subseteq M$ the set of points such that the fiber contains more than one point, we get
		\begin{align*}
			&\mathcal{H}^m(M)+\eta\mathcal{H}^m(E_j)
			\le\mass(\pi_*V_j)
			\le(1+C(M)d_j)\mass(V_j)
			\to\mathcal{H}^m(M),
		\end{align*}
		since $\pi$ is $(1+C(M)r)$-Lipschitz on $\bar B_r(M)$, and this forces $\mathcal{H}^m(E_j)\to 0$.
		
		Finally, the fact that each fiber is nonempty implies $d_H(\spt|V_j|,M)\to 0$. Thus, $V_j$ eventually satisfies all the assumptions of \cref{thm.core.main}, and we get $\spt|V_j|=M$ (eventually). However, this contradicts the previous remark that $\spt|S_j|\not\subseteq M$, which implies $\spt|V_j|\not\subseteq M$.
	\end{proof}
	
	We now turn to the proof of \cref{thm.core.main}, which is rather technical and is split into several steps.
	The main idea is to represent $V$ as a multigraph over $M$. This multigraph consists of a single layer above each $x\in M\setminus E$, in view of the last assumption.
	Hence, we can define a section $f:M\setminus E\to T^\perp M$ and, once we normalize it, we hope to obtain a Jacobi field in the limit $\delta\to 0$, defined a.e.\ on $M$ (recall that $\mathcal{H}^m(E)\le\delta$). However, in order to have a nontrivial limit section, we need a strong convergence, which is guaranteed if we have a uniform $W^{1,2}$ bound.
	
	In order to have this bound, we need to define a section on all of $M$. Since $V$ is not integer rectifiable, the multigraph is not a superposition of graphs of multiplicity one.
	Rather, roughly speaking, each layer could have some variable real-valued multiplicity, and we do not have any obvious information on the regularity of the latter.
	Hence, rather than taking an average of the points above each $x\in M$, we will perform a less canonical construction; it has the advantage of retaining enough regularity, and allows us to ignore the question of the regularity of these multiplicities.
	
	\begin{proof}[Proof of \cref{thm.core.main}]
	In the sequel, up to modifying $\eta$, we assume without loss of generality that $\Lambda=\mathcal{H}^m(M)$, so that $\mass(V)\le\mathcal{H}^m(M)$.
	Along the proof, we will often require that $\rho$ is suitably small (depending also on $\eta$); this is legitimate since we require $\delta<\rho$.
	
		\textbf{Step 1 (excess-displacement inequality).}
		For any $y\in U$, let $P^0_y\subseteq T_yN$ be the parallel transport of the $m$-plane $T_{\pi(y)}M$ along the shortest geodesic from $\pi(y)$ to $y$.
		Given an $m$-plane $P\subseteq T_yN$, with $y$ in $U$, we define its \emph{excess} to be
		\begin{align*}
			&\exc(P)^2:=\mz\|P-P^0_y\|^2,
		\end{align*}
		with respect to the Hilbert--Schmidt norm on endomorphisms of $T_yN$ (we identify a plane with the orthogonal projection onto it).
		Note that
		\begin{align*}
			&\exc(P)^2=\sum_{k=1}^m|(P^0_y)^\perp e_k|^2
		\end{align*}
		for any orthonormal basis $(e_k)_{k=1}^m$ of $P$.
		Also, given a Borel set $B\subseteq U$, we let
		\begin{align*}
			&\exc(V,B)^2:=\int_{\Gr_m(B)}\exc(P)^2\,dV(y,P)
		\end{align*}
		be the \emph{excess} of $V$ on $B$ (where $\Gr_m(B)$ the Grassmannian bundle of $m$-planes over $B$) and
		\begin{align*}
			&\dis(V,B)^2:=\int_B\operatorname{dist}(y,M)^2\,d|V|(y)
		\end{align*}
		its \emph{displacement}.
		
		
		Since $V$ is compactly supported in $U$, we can test its stationarity against the vector field $X:=\operatorname{grad}\frac{d_M^2}{2}$ (with $d_M:=\operatorname{dist}(\cdot,M)$). We claim that
		\begin{align*}
			&\dive_P X
			=\exc(P)^2+O(\exc(P) d_M(y))+O(d_M(y)^2)
		\end{align*}
		for any $m$-plane $P\subseteq T_y N$ and any $y\in U$.
		Indeed, the smooth map $P\mapsto\dive_P X$ (defined on $\Gr_m(U)$) vanishes with order two (at least) on each plane of the form $P=T_xM$,  $x\in M$, since $X$ vanishes on $M$ and $M$ is minimal. Also, at any $x\in M$, $\nabla X:T_xN\to T_xN$ is the orthogonal projection on the normal bundle; hence, $\dive_P X=\exc(P)^2$ for $P\subseteq T_x N$ and $x\in M$, and the claim follows.
		Hence, by stationarity of $V$,
		\begin{align*}
			&\int_{\Gr_m(U)}\exc(P)^2\,dV(y,P)
			\le C(M)\int_{\Gr_m(U)}(\exc(P) d_M(y)+d_M(y)^2)\,dV(y,P).
		\end{align*}
		By Young's inequality, we deduce
		\begin{align}\label{exc.dis}
			&\exc(V)
			\le C(M)\dis(V),
		\end{align}
		where we abbreviate $\exc(V):=\exc(V,U)$ and $\dis(V):=\dis(V,U)$.
		
		\textbf{Step 2 (construction of the section $\bm{f}$).}
		Given $\gamma>0$, we let $\tilde S$ be the subset of $\operatorname{spt}|V|\subseteq U$ consisting of those points $y$ such that
		\begin{align*}
			&\exc(V,B_r(y))^2
			\ge\gamma r^m
		\end{align*}
		for some radius $0<r<\mz\operatorname{inj}(M)$ (depending on $y$), or such that the approximate tangent plane of $V$ at $y$ does not exist.
		Also, let
		\begin{align*}
			&S:=\pi(\tilde S).
		\end{align*}
		
		In particular, Vitali's covering lemma gives
		\begin{align}\label{tilde.S.bound}
			&|V|(\tilde S)
			\le C(M,\gamma)\exc(V)^2,
		\end{align}
		since by the monotonicity formula we have $|V|(B_r(y))\le C(N)\mass(V)r^m\le C(M)r^m$ for all $y\in N$ and all $r>0$,
		as well as
		\begin{align}\label{S.bound}
			&\mathcal{H}^m(S)\le C(M,\gamma)\exc(V)^2,
		\end{align}
		since $\mathcal{H}^m(\pi(B_{r}(y)))\le C(M)r^m$.
		
		
		It is now convenient to identify $N$ isometrically as a submanifold of some Euclidean space $\R^L$. In the sequel, we will identify planes in $TN$ (such as $T_xM$ and $T_x^\perp M$) with planes in $\R^L$ (passing through the origin).
		
		Let $\mathcal{K}$ be the collection of all nonempty compact subsets of $\R^L$, endowed with the Hausdorff distance. The proof of the next lemma, modeled after \cite{Al},  is quite technical but essentially standard, and is postponed to the next subsection.

		\begin{lemmaen}\label{lip}
			Provided that $\rho$ and $\gamma$ are chosen small enough, depending on $M$ and $\eta$, the map given by
			\begin{align*}
				&F:M\setminus S\to\mathcal{K},\quad F(x):=\exp_x^{-1}(\spt|V|\cap\pi^{-1}(x))
			\end{align*}
			is locally 1-Lipschitz (namely, $1$-Lipschitz on any small geodesic ball $B_{c(M,\eta)}(x)$ in $M$). Also, for any $x\in M$, the cardinality of $\spt|V|\cap\pi^{-1}(x)\setminus\tilde S$ is bounded by a constant $C(M,\eta)$.
		\end{lemmaen}
		
		Note that $F(x)\subset T_x^\perp M\subset\R^L$; for instance, if $\spt|V|\cap\pi^{-1}(x)=\{x\}$, then $F(x)=\{0\}$. 		
		Also, denoting by $J(\pi,P)$ the Jacobian of $d\pi|_{P}$, for any $m$-plane $P\in\Gr_m(U)$ tangent to $y\in U$ it is easy to check that
		\begin{align}\label{J.approx}
			&J(\pi,P)
			=1-\mz\exc(P)^2+O(\exc(P)^3)+O(d_M(y)).
		\end{align}
		If $y\in\spt|V|\setminus\tilde S$, by the lower density bound we have
		\begin{align}\label{exc.point}
			&\exc(T_yV)^2
			\le\limsup_{r\to 0}\frac{\gamma r^m}{|V|(B_r(y))}
			\le C(M,\eta)\gamma,
		\end{align}
		and hence \cref{J.approx} gives
		\begin{align}\label{J.mz}
			&J(\pi,T_yV)\ge\mz,
		\end{align}
		provided that $\gamma$ and $\rho$ are small enough. The area formula, together with the second part of \cref{lip} (and the upper bound on the density given by the monotonicity formula), implies then that
		\begin{align*}
			&|V|(\pi^{-1}(S)\setminus\tilde S)
			\le C(M,\eta)\mathcal{H}^m(S).
		\end{align*}
		Using also \cref{exc.dis}, \cref{tilde.S.bound} and \cref{S.bound}, we conclude that
		\begin{align}\label{mass.bound}
			&|V|(\pi^{-1}(S))
			\le C(M,\eta,\gamma)\exc(V)^2
			\le C(M,\eta,\gamma)\dis(V)^2.
		\end{align}
		
		
		We now 
		assign to each $x\in M\setminus S$ the smallest rectangle $\prod_{k=1}^L[a_k,b_k]\subset\R^L$ containing $F(x)$. Since $F$ is locally $1$-Lipschitz,
		the same holds for $x\mapsto a_k(x)$ and $x\mapsto b_k(x)$. Hence, we can extend them to $C(M,\eta)$-Lipschitz maps from $M$ to $\R$, and define
		\begin{align*}
			&\hat f:M\to\R^L,\quad \hat f(x):=\Big(\frac{a_1(x)+b_1(x)}{2},\dots,\frac{a_{L}(x)+b_{L}(x)}{2}\Big),
		\end{align*}
		as well as the section of the normal bundle $T^\perp M$
		\begin{align*}
			&f:M\to T^\perp M,\quad f(x):={T_x^\perp M}(\hat f(x)),
		\end{align*}
		given by the orthogonal projection of $\hat f(x)$ onto $T_x^\perp M$ (viewed as an $(n-m)$-plane through the origin of $\R^L$).
		
		
		The maps $f$ and $\hat f$ are $C(M,\eta)$-Lipschitz and have the property that
		\begin{align}\label{hatf.diam}
			&|f(x)-\exp_x^{-1}(y)|
			\le|\hat f(x)-\exp_x^{-1}(y)|
			\le C(L)\operatorname{diam}(F(x)),
		\end{align}
		for all $x\in M\setminus S$ and $y\in\spt|V|\cap\pi^{-1}(x)$. 
		In particular, $f(x)=\hat f(x)=\exp_x^{-1}(y)$ if $x\nin S$ and $F(x)=\{y\}$ is a singleton.
		
		\textbf{Step 3 (bound on the differential of $\bm{f}$).}
		For any fixed $x\in M\setminus S$, we claim that
		\begin{align}\label{hausd}
			&\limsup_{x'\to x}\frac{d_H(F(x'),F(x))}{|x'-x|}
			\le C(M)\max_{y\in\spt|V|\cap\pi^{-1}(x)}(\exc(T_yV)+|y-x|)
		\end{align}
		as $x'\to x$ in $M\setminus S$.
		Indeed, given a sequence $x_k\to x$ (with $x_k\nin S$), let $y_k\in\spt|V|\cap\pi^{-1}(x_k)$ and assume that $y_k\to y\in\spt|V|\cap\pi^{-1}(x)$. Up to subsequences, assume that
		\begin{align*}
			&v:=\lim_{k\to\infty}\frac{x_k-x}{|x_k-x|},\quad w:=\lim_{k\to\infty}\frac{y_k-y}{|y_k-y|}
		\end{align*}
		both exist. By blowing up at scale $|y_k-y|$ and using the upper semicontinuity of the support, it is immediate to check that $w\in T_yV=:P_y$. By assumption, since $y\nin\tilde S$, \cref{exc.point} holds; hence, the vector $w$ is almost parallel to $P^0_y$, and thus to $T_xM$. This forces
		\begin{align*}
			&\mz|x_k-x|
			\le|y_k-y|
			\le 2|x_k-x|,
		\end{align*}
		provided that $\rho$ and $\gamma$ are small enough (and $k$ is big enough). Now $P_y^\perp w=0$ implies
		\begin{align*}
			&|(P_y^0)^\perp w|
			=|P_y^\perp w-(P_y^0)^\perp w|
			\le\|P_y-P_y^0\|,
		\end{align*}
		and thus 
		\begin{align*}
			&|T_x^\perp M (w)|
			\le\|P_y-P_y^0\|+C(M)|y-x|,
		\end{align*}
		where we used the fact that 
		\[
		\|P_y^0-T_xM\|\le C(M)|y-x|^2\le  C(M)|y-x|.
		\]
		Since $T_x^\perp M (v)=0$, letting $z:=\lim_{k\to\infty}\frac{y_k-y}{|x_k-x|}$ (up to subsequences) we get
		\begin{align*}
			&|T_x^\perp M(z-v)|
			\le 2|T_x^\perp M(w)|
			\le2\|P_y-P_y^0\|+C(M)|y-x|.
		\end{align*}
		On the other hand, $x_k-x=\pi(y_k)-\pi(y)$; dividing by $|x_k-x|$ and taking the limit, we get
		\begin{align*}
			&v=d\pi(y)[z]=T_xM(z)+O(|y-x|),
		\end{align*}
		which gives $T_xM(z-v)=O(|y-x|)$. Since $z-v\in T_xN$, we arrive at
		\begin{align*}
			\lim_{k\to\infty}\frac{|(y_k-x_k)-(y-x)|}{|x_k-x|}
			&=|z-v| \\
			&\le|T_xM(z-v)|+|T_x^\perp(z-v)| \\
			&\le C(M)(\exc(T_yV)+|y-x|).
		\end{align*}
		Writing $y=\exp_x(u)$ and $y_k=\exp_{x_k}(u_k)$, observe that
		\begin{align*}
			&(y_k-x_k-u_k)-(y-x-u)
			=O(|u||u_k-u|)+O(|u||x_k-x|)
		\end{align*}
		(by Taylor expansion of an extension of $\exp_x(u)-x-u$ to a map $H:M\times\R^L\to\R^L$
		such that $H(x,0)=0$ and $\de_{u}H(x,0)=0$). Since $|u|\le\rho$, for $\rho$ small enough this gives
		\begin{align*}
			&|(y_k-x_k)-(y-x)|
			\ge\mz|u_k-u|-C(M)|u||x_k-x|
		\end{align*}
		eventually, and (since $|u|\le C(M)|y-x|$) the previous bounds combine to give
		\begin{align*}
			&\limsup_{k\to\infty}\frac{|u_k-u|}{|x_k-x|}
			\le C(M)(\exc(T_yV)+|y-x|),
		\end{align*}
		which implies \cref{hausd}.
		
		From the construction of $\hat f$, we get the differential bound
		\begin{align}\label{hatf.meno.x}
			&|d\hat f(x)|
			\le C(M)\max_{y\in\spt|V|\cap\pi^{-1}(x)}(\exc(T_yV)+|y-x|)
		\end{align}
		for a.e.\ $x\in M\setminus S$, as well as
		\begin{align}\label{h}
			&|dh(x)|
			\le C(M)\max_{y\in\spt|V|\cap\pi^{-1}(x)}(\exc(T_yV)+|y-x|)
		\end{align}
		for the Lipschitz function
		\begin{align*}
			&h(x):=\operatorname{diam}F(x)
		\end{align*}
		(initially defined on $M\setminus S$ and then extended to a $C(M,\eta)$-Lipschitz function on $M$), for a.e.\ $x\in M\setminus S$.
		
		From \cref{hatf.meno.x} and the definition of $f$, it follows that
		\begin{align}\label{f.bd}
		\begin{aligned}
			|\nabla^\perp f(x)|
			&\le |df(x)| \\
			&\le |d\hat f(x)|+C(M)|\hat f(x)| \\
			&\le C(M)\max_{y\in\spt|V|\cap\pi^{-1}(x)}(\exc(T_yV)+|y-x|)
		\end{aligned}
		\end{align}
		a.e.\ on $M\setminus S$. 
		
		\textbf{Step 4 (definition and nontriviality of the limit section).}
		Assume by contradiction that the statement is false; hence, there exists a sequence of varifolds $V_j$ satisfying all the assumptions with $\delta=\delta_j$, for an infinitesimal sequence $\delta_j\to 0$, but such that $\spt|V_j|\neq M$. In the sequel, we write $F_j$, $f_j$ and $h_j$ to denote the functions $F$, $f$ and $h$ constructed above, for each $V=V_j$.
		
		Recall that, for $x\in M\setminus S$, the cardinality of $\spt|V|\cap\pi^{-1}(x)$ is at most $C(M,\eta)$.
		Integrating \cref{h} and \cref{f.bd}, using the area formula and the lower bound on the density, as well as \cref{S.bound} and \cref{exc.dis}, for the $C(M,\eta)$-Lipschitz functions $h_j$ and $f_j$ we get
		\begin{align*}
			&\int_M(|dh_j|^2+|\nabla^\perp f_j|^2)
			\le C\mathcal{H}^m(S_j)+C\exc(V_j)^2+C\dis(V_j)^2
			\le C\dis(V_j)^2
		\end{align*}
		for some $C=C(M,\eta,\gamma)$. From the bound
		$$|h_j(x)|+|f_j(x)|\le C(M)\max_{y\in\spt|V|\cap\pi^{-1}(x)}|y-x|$$
		on $M\setminus S_j$, we also obtain
		\begin{align*}
			&\int_{M} (h_j^2+|f_j|^2)
			\le C\mathcal{H}^m(S_j)+C\dis(V_j)^2
			\le C\dis(V_j)^2.
		\end{align*}
		
		Let
		\begin{align*}
			&\tilde f_j:=\frac{f_j}{\dis(V_j)},\quad \tilde h_j:=\frac{h_j}{\dis(V_j)}
		\end{align*}
		(note that $\dis(V_j)>0$, since otherwise $\spt|V_j|\subseteq M$ and the last assumption of the theorem gives $\spt|V_j|=M$, a contradiction).
		The sequences $(\tilde f_j)$ and $(\tilde h_j)$ are then bounded in $W^{1,2}(M)$.
		
		Since $\dis(V_j)\le\delta_j\mass(V_j)^{1/2}\to 0$, we have $\mathcal{H}^m(S_j)\to 0$ (by \cref{S.bound} and \cref{exc.dis}).
		Moreover, \(h_j=0\) on \(M\setminus (E_j\cup S_j)\) and \(\mathcal{H}^m(E_j)\to 0\) by assumption. Hence,
		\begin{align*}
			&\mathcal{H}^m(\{h_j\neq 0\})\le\mathcal{H}^m(S_j\cup E_j)\to 0.
		\end{align*}
		Any weak subsequential limit $\tilde h_\infty$ of $\tilde h_j$ is a strong limit in $L^2(M)$, and we deduce that $\tilde h_\infty=0$ (a.e.). Thus, $\tilde h_j\to 0$ in $L^2(M)$.
		
		Assume $\tilde f_j\weakto \tilde f_\infty$ in $W^{1,2}(M)$, up to subsequences. To conclude, we will show that $\tilde f_\infty$ is a nontrivial Jacobi field on $M$, giving the desired contradiction. Since
		\begin{align*}
			&|(y-x)-\exp_x^{-1}(y)|
			\le C(M)|y-x|^2
		\end{align*}
		for $y\in\pi^{-1}(x)$, from \cref{hatf.diam} we get
		\begin{align*}
			&|y-x|
			\le |f_j(x)|+C(M)h_j(x)+C(M)|y-x|^2.
		\end{align*}
		Hence, for all $x\in M\setminus S_j$ and $y\in\spt|V_j|\cap\pi^{-1}(x)$ we obtain
		\begin{align*}
			&|y-x|
			\le C(M)(|f_j(x)|+h_j(x)),
		\end{align*}
		provided that $\rho$ is small enough.
		Since \cref{lip} gives an upper bound on the cardinality of $\spt|V_j|\cap\pi^{-1}(x)$ (while, by the monotonicity formula, we have an upper bound on the density of $V_j$), the area formula and \cref{J.mz} give
		\begin{align*}
			&\dis(V_j)^2
			\le\delta_j^2|V_j|(\pi^{-1}(S_j))+C(M,\eta)\int_{M\setminus S_j}(|f_j|^2+h_j^2).
		\end{align*}
		Using \cref{mass.bound}, for $j$ big enough we conclude that
		\begin{align*}
			&\dis(V_j)^2
			\le C(M,\eta)\int_{M\setminus S_j}(|f_j|^2+h_j^2),
		\end{align*}
		which together with $\tilde h_j\to 0$ (in $L^2(M)$) implies that $\int_M \tilde f_j^2$ is bounded below by a positive constant.
		Thus, $\tilde f_\infty$ is nontrivial.
		
		\textbf{Step 5 (the limit section is a Jacobi field).}
		In this final step we linearize the stationarity of $V_j$.
		For $x\in M\setminus (E_j\cup S_j)$ we let $\theta_j(x)$ be the density of the (unique) point in $\spt|V_j|\cap\pi^{-1}(x)$.
		Assuming without loss of generality that $M$ is connected, note that $V_j\weakto\bar\theta M$ along a subsequence, for a constant $\eta\le\bar\theta\le 1$ (thanks to the constancy theorem).
		
		By applying Allard's strong constancy lemma locally on planes tangent to $M$ (see \cite[Theorem~1]{Al2} and the argument used in \cite[Theorem~2.2]{Al2}), we have
		$$\|\theta_j-\bar\theta\|_{L^1(M\setminus(S_j\cup E_j))}\to 0;$$
		hence, up to enlarging $E_j$ and $\delta_j$, we can assume that $|\theta_j-\bar\theta|\le\delta_j\to 0$ on $M\setminus E_j$, with $S_j\subseteq E_j$ and $\mathcal{H}^m(E_j)\le\delta_j$.
		
		Let $(y^1,\dots,y^n)$ be a local coordinate system defined on a neighborhood $U'\subseteq U$ of a point of $M$, with image $B_2^m\times B_2^{n-m}$, adapted in such a way that $M$ coincides with $B_2^m\times\{0\}$ in this chart, with $\pi(y)=(y^1,\dots,y^m,0,\dots)$, and with the metric $g_{ij}:=g(\de_i,\de_j)$ satisfying $g_{ij}=0$ along $M$, for $i\le m$ and $j\ge m+1$.
		Identifying $T^\perp M$ with $\R^{n-m}$, we also require that $\exp_x(v)=(x,v)$ for all $(x,v)\in B_2^m\times B_2^{n-m}$.
		
		Take a smooth section $\varphi:M\to T^\perp M$ of the normal bundle, fixed in the sequel, and extend it to a vector field on $U$ (still denoted $\varphi$) by parallel transport along geodesic rays orthogonal to $M$.
		
		For a $C(M,\eta)$-Lipschitz function $f:B_1^m\to B_1^{n-m}$ 
		and a bounded multiplicity $\theta:B_1^m\to[0,\infty)$, 
		the varifold $\Gamma$ on $U'$ given by the graph of $f$ (with multiplicity $\theta$) has
		\begin{align*}
			&\int_{\Gr_m(U)}\dive_P\varphi\,d\Gamma(y,P) \\
			&=\int_{B_1^m} \theta(x)\det(g^f(x))^{1/2} (g^f(x))^{ij}
			g_{(x,f(x))}(\nabla_{\de_i+\de_i f(x)}\varphi(x,f(x)),\de_j+\de_j f(x))\,dx,
		\end{align*}
		where we denote by
		\begin{align*}
			&(g^f(x))_{ij}:=g_{(x,f(x))}(\de_i+\de_i f(x),\de_j+\de_j f(x))
		\end{align*}
		the metric of the graph (viewing $f(x)$ as a vector in $\{0\}\times\R^{n-m}$) and we sum over all $i,j=1,\dots,m$. Here $\nabla$ is the Levi-Civita connection on $TN$ for the Riemannian metric $g$.
		
		We observe that $(g^f(x))_{ij}=g_{ij}(x,f(x))+O(|df|^2)+O(|f||df|)$, by our assumptions on the coordinate system, while
		\begin{align*}
			&\nabla_{\de_i+\de_i f(x)}\varphi(x,f(x))
			=\nabla_{\de_i}\varphi(x,f(x))+O(|f||df|),
		\end{align*}
		since on $M$ the vector field $\varphi$ is parallel along all directions of $T^\perp M$.
		
		Also, since $\varphi$ is orthogonal to $\de_j$ on $M$, we have
		\begin{align*}
			&g^{ij}(x,0)g_{(x,0)}(\nabla_{\de_i}\varphi,\de_j)
			=-g^{ij}(x,0)g_{(x,0)}(\varphi,\nabla_{\de_i}\de_j)
			=0
		\end{align*}
		by minimality of $M$.
		Thus, we obtain
		\begin{align*}
			&\int_{\Gr_m(U)}\dive_P\varphi\,d\Gamma(y,P)
			=\int_{B_1^m} \theta(L(f,\varphi)+O(|f|^2+|df|^2)),
		\end{align*}
		with the bilinear operator
		\begin{align*}
			L(f,\varphi)(x)&:=g^{ij}(x,0)g_{(x,0)}(\nabla_{\de_i}\varphi(x,0),\de_j f(x)) \\
			&\quad\,+\sum_{k=m+1}^n \de_{v^k}[\det(g(x,v))^{1/2}g^{ij}(x,v)g_{(x,v)}(\nabla_{\de_i}\varphi(x,v),\de_j)]\Big|_{v=0}\,f^k(x).
		\end{align*}
		
		Note that $|df(x)|\le C(M)(\exc(P)+f(x))$, with $P\subseteq T_{(x,f(x))}N$ the plane spanned by $\{\de_i+\de_i f(x)\}_{i=1}^m$.
		
		Assuming that $\varphi|_M$ is supported on $B_1^m\times\{0\}\subset M\cap U'$, for $V_j$ we have
		\begin{align}\label{stat}
			&\int_{\Gr_m(U)}\dive_P\varphi\,dV_j(y,P)=0
		\end{align}
		by stationarity. We claim that
		\begin{align}\label{jac.claim}
			&\int_{M\cap U'}L(f_j,\varphi)=o(\dis(V_j))
		\end{align}
		as $j\to\infty$. This implies $\int_{M\cap U'}L(\tilde f_j,\varphi)\to 0$ and thus
		\begin{align}\label{jac.concl}
			&\int_{M\cap U'}L(\tilde f_\infty,\varphi)=0.
		\end{align}
		
		In the chosen chart, $f_j|_{B_1^m}$ is eventually a $C(M,\eta)$-Lipschitz function with values into $B_1^{n-m}$. By definition of $f_j$, for any $x\in B_1^m\setminus E_j$ we have that
		$(x,f_j(x))$ is the unique point in $\spt|V_j|\cap\pi^{-1}(x)\subseteq\{x\}\times B_2^{n-m}$.
		
		Recalling \cref{f.bd} and \cref{exc.dis}, the previous computation implies that
		\begin{align}\label{linear.bd}
			&\int_{\Gr_m(U\setminus\pi^{-1}(E_j))}\dive_P\varphi\,dV_j
			=\int_{M\cap U'\setminus E_j}\theta_j L(f_j,\varphi)+O(\dis(V_j)^2).
		\end{align}
		Also, for a plane $P\subseteq T_yN$ we have $\dive_P\varphi=O(d_M(y)+\exc(P))$,
		as the divergence vanishes when $y\in M$ and $P$ is tangent to $M$ (by minimality of $M$).
		Using Cauchy--Schwarz we deduce that
		\begin{align*}
			&\int_{\Gr_m(\pi^{-1}(E_j))}\dive_P\varphi\,dV_j(y,P)
			=O(|V_j|(\pi^{-1}(E_j))^{1/2})O(\dis(V_j)).
		\end{align*}
		By \cref{lip} and \cref{J.mz}, the measure $|V_j|(\pi^{-1}(E_j\setminus S_j))$ is in turn $O(\mathcal{H}^m(E_j))$, which is infinitesimal as $j\to\infty$.
		Finally, by \cref{mass.bound}, $|V|(\pi^{-1}(S_j))=O(\dis(V_j)^2)$. Hence, we obtain
		\begin{align*}
			&\int_{\Gr_m(\pi^{-1}(E_j))}\dive_P\varphi\,dV_j
			=o(\dis(V_j)).
		\end{align*}
		Together with \cref{stat} and \cref{linear.bd}, we arrive at
		\begin{align}\label{jacobi.almost}
			&0=\int_{\Gr_m(U)}\dive_P\varphi\,dV_j(y,P)
			=\int_{M\cap U'\setminus E_j}\theta_j L(f_j,\varphi)+o(\dis(V_j)).
		\end{align}
		
%
		
		Recalling that $|\theta_j-\bar\theta|\le\delta_j$ on $M\setminus E_j$, \cref{jacobi.almost} becomes
		\begin{align*}
			\int_{M\cap U'} \bar\theta L(f_j,\varphi)
			&=O(\mathcal{H}^m(E_j)^{1/2}\|f_j\|_{W^{1,2}})+o(\|f_j\|_{W^{1,2}})+o(\dis(V_j)) \\
			&=o(\dis(V_j)),
		\end{align*}
		where we used again Cauchy--Schwarz.
		This proves our claim \cref{jac.concl}.
		
		This implies that $\tilde f_\infty$ is a Jacobi field: to check this,
		note that
		\begin{align*}
			&\frac{d^2}{ds\,dt}\mathcal{H}^m(M_{s\psi+t\varphi})\Big|_{s,t=0}
			=J(\psi,\varphi)
		\end{align*}
		for any smooth section $\psi:M\to T^\perp M$ (using the notation from \cref{sec.prel}),
		since the left-hand side is symmetric and specializes to $\frac{d^2}{dt^2}\mathcal{H}^m(M_{t\varphi})\Big|_{t=0}$ for $\psi=\varphi$.
		Also, calling $\Phi_t^\varphi$ the flow of $\varphi$ (defined for small times, near $M$), we have
		\begin{align*}
			&\Phi_{-t}^{\varphi}(M_{s\psi+t\varphi})=M_{s\psi+O(st)}
		\end{align*}
		for some smooth remainder $O(st)$, since $\Phi_{-t}^\varphi(M_{t\varphi})=M=M_0$; hence,
		\begin{align*}
			&\frac{d}{dt}\mathcal{H}^m(M_{s\psi+t\varphi})\Big|_{t=0}
			=\int_{M_{s\psi}}\dive_{M_{s\psi}}\varphi
			+\frac{d}{dt}\mathcal{H}^m(M_{s\psi+O(st)})\Big|_{t=0}.
		\end{align*}
		Since $M=M_{0}$ is minimal, the second term is $O(s^2)$. Thus,
		\begin{align*}
			&J(\psi,\varphi)
			=\frac{d}{ds}\int_{M_{s\psi}}\dive_{M_{s\psi}}\varphi\Big|_{s=0}.
		\end{align*}
		With $\varphi$ as above, the same computations used earlier show that
		\begin{align*}
			&\int_{M_{s\psi}}\dive_{M_{s\psi}}\varphi
			=\int_{M\cap U'}L(s\psi,\varphi)+O(s^2);
		\end{align*}
		we deduce that
		\begin{align*}
			&J(\psi,\varphi)
			=\int_{M\cap U'} L(\psi,\varphi)
		\end{align*}
		and, by approximation, we can take $\psi:=\tilde f_\infty$, obtaining $J(\tilde f_\infty,\varphi)=0$.
		Since $\varphi$ was arbitrary, we obtain that $J\tilde f_\infty=0$ in the weak sense.
		By standard elliptic regularity $\tilde f_\infty$ is smooth, and it is then a nontrivial Jacobi field. This is the desired contradiction, since $M$ was assumed to be non-degenerate.
	\end{proof}
	
	\subsection{Proof of \cref{lip}}
	In this technical subsection we show how to obtain \cref{lip}. We first state two useful facts. In the sequel, we let
	\begin{align*}
		&\mathcal{C}_{r,s}:=B_r^m\times B_s^{n-m}
	\end{align*}
	and, for varifolds $V$ on an open subset $\Omega\subseteq\R^n$, we denote
	\begin{align*}
		&\exc(V,\Omega)^2:=\int_{\Gr_m(\Omega)}\mz\|P-\bar P\|^2\,dV(y,P),
	\end{align*}
	where $\bar P$ is the plane spanned by $e_1,\dots,e_m$. Also, given $\Lambda>0$, we fix a metric $d_{\Omega,\Lambda}$ inducing the weak topology on the set $\mathcal{V}_{\Omega,\Lambda}$ of all varifolds in $\Omega$ with mass at most $\Lambda$, and we call $\mathcal{S}_{\Omega,\Lambda}\subseteq\mathcal{V}_{\Omega,\Lambda}$ the subset of stationary $m$-varifolds. We will omit $\Omega$ when it is clear from the context.
	
	\begin{lemmaen}\label{tech1}
	Given $\Lambda,\bar\eta>0$ and $\tau\in(0,1)$, there exists $\delta(n,\Lambda,\bar\eta,\tau)>0$ with the following property. If an $m$-varifold $V$ on $\Omega=\mathcal{C}_{2,2}$ has
	\begin{itemize}
	\item $|V|(B_r^n(y))\ge\omega_m\bar\eta r^m$ for all $y\in\spt|V|$ and $0<r<d(y,\de\Omega)$,
	\item mass $|V|(\mathcal{C}_{2,2})\le\Lambda$,
	\item excess $\exc(V,\mathcal{C}_{2,2})\le\delta$,
	\item $d_\Lambda(V,\mathcal{S}_\Lambda)\le\delta$,
	\item $\spt|V|\cap(S^{m-1}\times\{0\})\neq\emptyset$,
	\end{itemize}
	then $|V|(\mathcal{C}_{1/2,\tau})>0$.
	\end{lemmaen}
	
	\begin{lemmaen}\label{tech2}
	Given $\bar R\ge 1$ and $\Lambda,\bar\eta,\tau>0$, there exists $\delta(n,\bar R,\Lambda,\bar\eta,\tau)>0$ with the following property. Given $1\le S\le R\le\bar R$ and an $m$-varifold $V$ on $\Omega=B_{4R}^n$ having
	\begin{itemize}
	\item $|V|(B_r^n(y))\ge\omega_m\bar\eta r^m$ for all $y\in\spt|V|$ and $0<r<d(y,\de\Omega)$,
	\item mass $|V|(B_{4R}^n)\le\Lambda$,
	\item excess $\exc(V,B_{4R}^n)\le\delta$,
	\item $d_\Lambda(V,\mathcal{S}_\Lambda)\le\delta$,
	\item $0\in\spt|V|$,
	\end{itemize}
	we have $|V|(\mathcal{C}_{R,3S})>(1-\tau)R^m|V|(\bar{\mathcal{C}}_{1,2S})$.
	\end{lemmaen}
	
	The proof is in both cases an easy argument by compactness and contradiction, based on the fact that a stationary $m$-varifold $V'$ on a convex open set $\Omega\subseteq\R^n$ with $\exc(V',\Omega)=0$ and density at least $\bar\eta$ is just a locally finite union of planes (intersected with $\Omega$), each with a constant density. To check this, note that for any function $\varphi:\R^n\to\R$ depending only on the last $n-m$ coordinates we have
	\begin{align*}
		&\delta (\varphi V')[X]=\delta V'[\varphi X]=0,
	\end{align*}
	for any smooth vector field with compact support in $\Omega$. Also,
	\begin{align*}
		&|\varphi V'|(B_r^n(y))\ge\omega_m\Theta^m(|V'|,y)r^m\ge\omega_m\bar\eta r^m
	\end{align*}
	if $\varphi(y)=1$ (by the monotonicity formula). Hence, if $y=(x,t)\in\spt|V'|$, then by approximation we get that $\uno_{\R^m\times\{t\}}V'$ is stationary and has $y$ in its support.
	By the constancy theorem, this varifold equals a constant positive multiple of the plane $\R^m\times\{t\}$, with density at least $\bar\eta$.
	
	
	We will apply these lemmas to suitable rescalings $\bar V_{a,r}$ of a given varifold $\bar V$ in $\R^n$, where
	\begin{align*}
		&\bar V_{a,r}:=(T_{a,r})_*\bar V,\quad T_{a,r}(y):=\frac{y-a}{r}.
	\end{align*}
	
	With these lemmas in hand, the proof of \cref{lip} becomes quite straightforward.
	
	\begin{proof}[Proof of \cref{lip}]
		Given a point $x_0\in M$, we fix a chart $\phi:U'\to \mathcal{C}_{r_0,r_0}$ centered at $x_0$ (with $U'\subseteq U$), in such a way that $\phi(U'\cap M)=B_{r_0}^m\times\{0\}$ and
		\begin{align*}
			&\phi^{-1}(x,t)=\exp_{\phi^{-1}(x,0)}(t^ie_i(x))
		\end{align*}
		(for all $x\in B_{r_0}^m$ and $t\in B_{r_0}^{n-m}$, with an implicit sum over $i=1,\dots,n-m$),
		where $(e_i(x))_{i=1}^{n-m}$ is an orthonormal basis of $T_{\phi^{-1}(x,0)}^\perp M$. We can also require that $g_{ij}(0)=\delta_{ij}$, $c(M)\delta\le g\le C(M)\delta$, and $\|g\|_{C^2}\le C(M)$.
		
		We denote by $\bar\pi:\R^n\to\R^m\times\{0\}=\R^m$ the orthogonal projection (which corresponds to $\pi:U\to M$ in the coordinate chart) and let $\bar V$ be the image under $\phi$ of the varifold $V$ (restricted to $U'$), so that $\bar V$ is a rectifiable $m$-varifold in $\mathcal{C}_{r_0,r_0}$. The monotonicity formula for $V$ gives
		\begin{align}\label{ball.bds}
			&\omega_m\bar\eta r^m\le|\bar V|(B_{r}^n(y))\le C(M) r^m
		\end{align}
		for all balls $B_r^n(y)\subseteq\mathcal{C}_{r_0,r_0}$ with center $y\in\spt|\bar V|$, for some $\bar\eta=c(M)\eta$.
		Also, choosing $r_0$ small enough (depending on a given $0<\tau<1$), we can guarantee that
		\begin{align}\label{mono.tau}
			&|\bar V|(B_r^n(y))\ge(1-\tau)\Big(\frac{r}{s}\Big)^m|\bar V|(B_s^n(y))
		\end{align}
		for all balls $B_r^n(y)\subseteq\mathcal{C}_{r_0,r_0}$ and $0<s<r$.
		
		Note that, for $\rho<r_0$, from the assumption $\spt|V|\subseteq U=B_\rho(M)$ we deduce
		\begin{align*}
			&\spt|\bar V|\subseteq B_{r_0}^m\times B_{\rho}^{n-m}.
		\end{align*}
		In the sequel, we will assume that $\rho$ is much smaller than $r_0$.
		
		We now show the first conclusion. Assume that $x_0\nin S$, fix $x\in B_\rho^m\setminus\{0\}$ and take any $y=(x,t)\in\spt|\bar V|\cap\bar\pi^{-1}(x)$.
		
		Provided that $\rho$ and $\gamma$ are small enough, we can apply \cref{tech1} (for a fixed $\tau$) to the varifold
		$\bar V_{(0,t),|x|}$. Note that the upper bound on its mass, with $\Lambda=\Lambda(M)$, is guaranteed by the monotonicity formula for $V$,
		while the smallness of the excess comes from $x_0\nin S$, and the last condition 
		follows from the fact that $(x,t)\in\spt|\bar V|$.
		
		Also, we claim that $\bar V_{a,r}$ is $\delta$-close to the set $\mathcal{S}_\Lambda$ of stationary varifolds on $\mathcal{C}_{2,2}$ (for all $a\in\mathcal{C}_{r_0/2,r_0/2}$ and all $0<r<r_0/4$), provided that $r_0$ is small enough (depending on $\delta$): indeed, for any vector field $X\in C^\infty_c(\mathcal{C}_{2,2},\R^n)$ and any plane $P\subseteq T_y\R^n$, our assumptions on $(g_{ij})$ give
		\begin{align*}
			&\lvert\dive_P^{g'} X-\dive_P^\delta X\rvert
			\le C(M)r_0\|X\|_{C^1},
		\end{align*}
		where $\dive_P^\delta$ denotes the flat divergence along $P$, while $\dive_P^{g'}$ is the one with respect to the rescaled metric $g':=r^{-2}(T_{a,r}^{-1})^*g$. Thus, the claim follows from the stationarity of $\bar V_{a,r}$ with respect to $g'$ and a standard compactness argument.
		
		From \cref{tech1} (applied to $\bar V_{(0,t),|x|}$) we deduce that there exists
		\begin{align*}
			&y_1=(x_1,t_1)\in\spt|\bar V|\quad\text{with }|x_1|\le\mz|x|,\ |t_1-t|\le\tau|x|.
		\end{align*}
		Iterating, we get a sequence $y_k=(x_k,t_k)$ such that $|x_{k+1}|\le\mz|x_k|$ and $|t_{k+1}-t_k|\le\tau|x_k|$.
		In the limit we get a point $(0,t_\infty)\in\spt|\bar V|$ such that $|t_\infty-t|\le 2\tau|x|$.
		
		This implies that \(F(x)\) is contained in a \((2\tau+C(M)\rho)|x|\)-neighborhood of \(F(0)\) (here the term \(O(\rho)\) comes from the change of coordinates); by symmetry  it easily follows that the map $F$ is locally $(2\tau+C(M)\rho)$-Lipschitz on $M\setminus S$. Hence, $x\mapsto F(x)$ is locally $1$-Lipschitz once we take $\tau$ and $\rho$ (as well as $\gamma$) to be small enough.
		
		We now turn to the second conclusion. Let $\Lambda=\Lambda(M)$ be such that $|\bar V|(B_r^n(y))\le\Lambda r^n$ for all $y\in\mathcal{C}_{r_0,r_0}$ and all $r>0$.
		We will show that we cannot have a set $A:=\{t_1,\dots,t_K\}\subset B_\rho^{n-m}$ of distinct values such that $(0,t_k)\in\spt|\bar V|\setminus\tilde S$ for all $k=1,\dots,K$, with $K$ the smallest integer greater or equal than $\frac{\Lambda}{\omega_m\bar\eta}+1$.
		
		Suppose by contradiction that this holds. Fix $\tau>0$ small and let
		\begin{align*}
			&r:=\mz\min_{k\neq k'}|t_k-t_{k'}|\le \rho.
		\end{align*}
		By \cref{ball.bds} we have
		\begin{align*}
			&|\bar V|((0,t_k)+\mathcal{C}_{r,r})
			\ge|\bar V|(B_r^n(0,t_k))
			\ge\omega_m\bar\eta r^m
			\ge\tau r^m
		\end{align*}
		for all $k=1,\dots,K$. We start from $S:=1$ and we keep replacing $S$ with $6S$ until
		\begin{align*}
			&B_{6Sr}^{n-m}(t_k)\cap A=B_{Sr}^{n-m}(t_k)\cap A
		\end{align*}
		for all $k$. Observe that $S\le C(K)=C(M,\eta)$.
		
		Now we take a maximal collection of balls $B_{3Sr}^{n-m}(t_k)$, indexed by a subset $E\subseteq\{1,\dots,K\}$. Note that $A\subseteq\bigcup_{k\in E} B_{Sr}^{n-m}(t_k)$, since for any index $\ell$ the ball $B_{3Sr}^{n-m}(t_{\ell})$ intersects $B_{3Sr}^{n-m}(t_k)$ for some $k\in E$, so $t_{\ell}$ belongs to $B_{6Sr}^{n-m}(t_k)$, and hence to $B_{Sr}^{n-m}(t_k)$ by our choice of $S$. It follows that
		\begin{align*}
			&\sum_{k\in E}|\bar V|((0,t_k)+\mathcal{C}_{r,2Sr})
			\ge\sum_{\ell}|\bar V|((0,t_\ell)+\mathcal{C}_{r,r})
			\ge K\omega_m\bar\eta r^m.
		\end{align*}
		
		Fix now $3S\le R\le C(M,\eta,\tau)$ so big that
		\begin{align}\label{def.R}
		&R^m
		\ge(1-\tau)(R^2+9S^2)^{m/2}.
		\end{align}
		We can apply \cref{tech2} to $\bar V_{(0,t_k),r}$ and find
		\begin{align}\label{gain}
		\begin{aligned}
			|\bar V|((0,t_k)+\mathcal{C}_{Rr,3Sr})
			&\ge(1-\tau)R^m|\bar V|((0,t_k)+\mathcal{C}_{r,2Sr}) \\
			&\ge(1-\tau)R^m\sum_{\ell\in B_{Sr}^{n-m}(t_k)\cap A}|\bar V|((0,t_\ell)+\mathcal{C}_{r,r}) \\
			&\ge(1-\tau)|B_{Sr}^{n-m}(t_k)\cap A|\omega_m\bar\eta(Rr)^m
		\end{aligned}
		\end{align}
		for any $k\in E$
		(provided that $\rho$ and $\gamma$ are small enough).
		
		If $E=\{k\}$ is a singleton, then we are done since
		\begin{align*}
			\Lambda (Rr)^m
			&\ge (1-\tau)\Lambda((R^2+9S^2)^{1/2}r)^m \\
			&\ge (1-\tau)|\bar V|(B_{(R^2+9S^2)^{1/2}r}^n(0,t_k)) \\
			&\ge (1-\tau)|\bar V|((0,t_k)+\mathcal{C}_{Rr,3Sr}) \\
			&\ge (1-\tau)^2 K\omega_m\bar\eta(Rr)^m.
		\end{align*}
		This contradicts the fact that $K\ge\frac{\Lambda}{\omega_m\bar\eta}+1$, once we choose $\tau$ small enough.
		
		Otherwise, we iterate the procedure. Let $r_1:=r$, $S_1:=S$, $R_1:=R$, $E_1:=E$, and define
		\begin{align*}
			&r_2:=\mz\min_{k,k'\in E_1,\,k\neq k'}|t_k-t_{k'}|.
		\end{align*}
		Note that $r_2\ge 3S_1r_1$, by definition of $S$. We claim that
		$$ |\bar V|((0,t_k)+\mathcal{C}_{r_2,r_2}) \ge (1-\tau)^3\Big(\frac{r_2}{r_1}\Big)^m|\bar V|((0,t_k)+\mathcal{C}_{r_1,2S_1r_1}) $$
		for all $k\in E$. Indeed, if $(R_1^2+9S_1^2)^{1/2}r_1\ge r_2$, then ${r_2}/{r_1}\le C(M,\eta,\tau)$ and we can apply \cref{tech2} to $\bar V_{(0,t_k),r_1}$ (with $r_2/r_1$ in place of $R$) in order to get
		\begin{align*}
			|\bar V|((0,t_k)+\mathcal{C}_{r_2,r_2})
			&\ge |\bar V|((0,t_k)+\mathcal{C}_{r_2,3S_1r_1}) \\
			&\ge (1-\tau)\Big(\frac{r_2}{r_1}\Big)^m|\bar V|((0,t_k)+\mathcal{C}_{r_1,2S_1r_1}).
		\end{align*}
		If instead $r_2\ge(R_1^2+9S_1^2)^{1/2}r_1$, we apply \cref{mono.tau} to obtain
		\begin{align*}
			|\bar V|((0,t_k)+\mathcal{C}_{r_2,r_2})
			&\ge |\bar V|(B_{r_2}^n(0,t_k)) \\
			&\ge (1-\tau)\Big(\frac{r_2}{(R_1^2+9S_1^2)^{1/2}r_1}\Big)^m|\bar V|(B_{(R_1^2+9S_1^2)^{1/2}r_1}^n(0,t_k)) \\
			&\ge (1-\tau)^2\Big(\frac{r_2}{R_1r_1}\Big)^m|\bar V|((0,t_k)+\mathcal{C}_{R_1r_1,3S_1r_1}) \\
			&\ge (1-\tau)^3\Big(\frac{r_2}{r_1}\Big)^m|\bar V|((0,t_k)+\mathcal{C}_{r_1,2S_1r_1}),
		\end{align*}
		where we also used \cref{def.R} and \cref{gain}.
		In both cases, the claim follows. We then obtain
		\begin{align*}
			&|\bar V|((0,t_k)+\mathcal{C}_{r_2,r_2})
			\ge (1-\tau)^3|B_{S_1r_1}^{n-m}(t_k)\cap A|\omega_m\bar\eta r_2^m.
		\end{align*}
		
		Let $1\le S_2\le C(M,\eta)$ be such that
		\begin{align*}
			&B_{6S_2r_2}^{n-m}(t_k)\cap A=B_{S_2r_2}^{n-m}(t_k)\cap A
		\end{align*}
		for all $k\in E_1$.
		We select $E_2\subseteq E_1$ in the same way as $E_1$, and again we have
		\begin{align*}
			&\sum_{k\in E_2}|\bar V|((0,t_k)+\mathcal{C}_{r_2,2S_2r_2})
			\ge (1-\tau)^3 K\omega_m\bar\eta r_2^m.
		\end{align*}
		If $E_2$ is a singleton, then we reach a contradiction as before.
		Otherwise, we keep iterating. Note that $|E_1|>|E_2|>\cdots$, since at every step at least two clusters of points $(0,t_k)$ merge together; hence, at most $K$ iterations are required. This completes the proof.
	\end{proof}

\end{document}